\documentclass[smallcondensed,final]{svjour3}
\usepackage{amsmath,amssymb,amsfonts,stmaryrd}
\usepackage{subfigure}
\usepackage{graphicx}
\usepackage{tikz, pgfplots}
\usepackage{hhline}
\usepackage{float}
\usepackage{multirow}
\usepackage{appendix}
\usepackage{placeins}
\usepackage{xparse}
\usepackage{mathtools}
\mathtoolsset{showonlyrefs=true}
\usepackage{fixltx2e}
\MakeRobust{\eqref}
\usepackage{everypage}
\usepackage{xspace}
\usepackage{ifthen}
\usepackage{enumitem}
\usepackage{url}
\reversemarginpar
\usepackage{algorithm,algpseudocode}
\definecolor{notes}{RGB}{0,255,178}
\usepackage[color=notes]{todonotes}
\usepackage{array}

\usepackage[final]{changes}
\usepackage{lipsum}
\colorlet{Changes@Color}{red}

\makeatletter
\def\@tvsp{\mathchoice{{}\mkern-3mu}{{}\mkern-3mu}{{}\mkern-3mu}{}}
\def\ltrivert{|\@tvsp|\@tvsp|}
\def\rtrivert{|\@tvsp|\@tvsp|}
\makeatother
\makeatletter
\def\@avgsp{\mathchoice{{}\mkern-6mu}{{}\mkern-6mu}{{}\mkern-6mu}{}}
\def\llaverage{\{\@avgsp\{}
\def\rraverage{\}\@avgsp\}}
\makeatother
\NewDocumentCommand{\normDG}{O{\cdot} O{j} O{}}{\ensuremath{\|\ifthenelse{\equal{#1}{}}{\cdot}{#1}\|_{DG,\ifthenelse{\equal{#2}{}}{j}{#2}}\ifthenelse{\equal{#3}{}}{}{^{#3}}}\xspace} 

\NewDocumentCommand{\normL}{O{\cdot} O{2} O{\Om} O{}}{\ensuremath{\|\ifthenelse{\equal{#1}{}}{\cdot}{#1}\|_{L^{\ifthenelse{\equal{#2}{}}{1}{#2}}(\ifthenelse{\equal{#3}{}}{\Om}{#3})}\ifthenelse{\equal{#4}{}}{}{^{#4}}}\xspace}

\NewDocumentCommand{\normH}{O{\cdot} O{1} O{\Om} O{}}{\ensuremath{\|\ifthenelse{\equal{#1}{}}{\cdot}{#1}\|_{H^{\ifthenelse{\equal{#2}{}}{1}{#2}}(\ifthenelse{\equal{#3}{}}{\Om}{#3})}\ifthenelse{\equal{#4}{}}{}{^{#4}}}\xspace}

\NewDocumentCommand{\norms}{O{\cdot} O{s} O{j} O{}}{\ensuremath{\ltrivert\ifthenelse{\equal{#1}{}}{\cdot}{#1}\|_{\ifthenelse{\equal{#2}{}}{s}{#2},\ifthenelse{\equal{#3}{}}{j}{#3}}\ifthenelse{\equal{#4}{}}{}{^{#4}}}\xspace}

\NewDocumentCommand{\Aa}{O{j} O{\cdot} O{\cdot}}{\ensuremath{\mathcal{A}_{\ifthenelse{\equal{#1}{}}{j}{#1}}(\ifthenelse{\equal{#2}{}}{\cdot}{#2},\ifthenelse{\equal{#3}{}}{\cdot}{#3})}\xspace} 

\NewDocumentCommand{\mcal}{O{} O{} O{}}{\ensuremath{\mathcal{#1}\ifthenelse{\equal{#2}{}}{}{_{#2}}\ifthenelse{\equal{#3}{}}{}{^{#3}}}\xspace}

\newcommand{\average}[1]{\ensuremath{\llaverage #1\rraverage}\xspace}  

\newcommand{\jump}[1]{\ensuremath{\llbracket #1\rrbracket}\xspace}

\newcommand{\Om}{\ensuremath{\Omega}\xspace}

\newcommand{\elem}{\ensuremath{\kappa}\xspace}

\newcommand{\h}[1]{\ensuremath{h_{#1}}\xspace}
\newcommand{\hmin}[1]{\ensuremath{h_{#1}}\xspace}

\newcommand{\abs}[1]{| #1 |}
\newcommand{\CINV}[4]{\textcolor{black}{ \mathsf{C}_{\mathsf{INV}}(#1,#2, #3, #4)}} 
\newcommand{\CINVnew}[1]{\textcolor{black}{ \mathsf{C}^{#1}_{\mathsf{inv}}}}

\newcommand{\CInew}[1]{\textcolor{black}{\mathsf{C}_{\mathsf{I}}^{#1}}}

\newcommand{\Cinterpnew}[1]{\textcolor{black}{\mathsf{C}^{#1}_{\mathsf{interp}}}} 
\newcommand{\Ccont}{\textcolor{black}{\mathsf{C}_{\mathsf{cont}}}}
\newcommand{\Ccoer}{\textcolor{black}{\mathsf{C}_{\mathsf{coer}}}}

\newcommand{\CGnew}[1]{\textcolor{black}{\mathsf{G}^{#1}}}

\newcommand{\CLtwonew}[1]{\textcolor{black}{\mathsf{C}^{#1}_{L^2}}}
\newcommand{\Cstab}{\textcolor{black}{\mathsf{C}_{\mathsf{stab}}}}
\newcommand{\Cmesh}{\textcolor{black}{\mathsf{C}_{\mathsf{mesh}}}}
\newcommand{\Ceig}[1]{\textcolor{black}{\mathsf{C}^{#1}_{\mathsf{eig}}}}
\newcommand{\Ctwolvl}{\textcolor{black}{\mathsf{C}_{\mathsf{2lvl}}}}
\newcommand{\Cequiv}{\textcolor{black}{\mathsf{C}_{\mathsf{equiv}}}}
\newcommand{\Chat}{\textcolor{black}{\widehat{\mathsf{C}}}}

\newcommand{\Ctildenew}[2]{\textcolor{black}{\widetilde{\mathsf{C}}_{#1,#2}}}

\newcommand{\NPCG}{\textcolor{black}{\mathsf{N}_{\mathsf{it}}^{\mathsf{PCG}}}}
\newcommand{\NCG}{\textcolor{black}{\mathsf{N}_{\mathsf{it}}^{\mathsf{CG}}}}
\newcommand{\NAMG}{\textcolor{black}{\mathsf{N}_{\mathsf{it}}^{\mathsf{AMG}}}}

\newcommand{\paul}[1]{\textcolor{black}{#1}}

\newtheorem{assumption}[theorem]{Assumption}

\title{Multigrid algorithms for $\boldsymbol{hp}$-version Interior Penalty Discontinuous Galerkin methods on polygonal and polyhedral meshes\thanks{Paola F. Antonietti has been partially supported by SIR (Scientific Independence of young Researchers) starting grant n. RBSI14VT0S ``\emph{PolyPDEs: Non-conforming polyhedral finite element methods for the approximation of partial differential equations}" funded by the Italian Ministry of Education, Universities and Research (MIUR).}}

\titlerunning{Multigrid algorithms for $hp$--DG methods on polygonal and polyhedral meshes}

\author{P. F.~Antonietti \and P.~Houston \and X.~Hu \and M.~Sarti \and M.~Verani}
\institute{P. F. Antonietti \and M. Sarti \and M. Verani\at MOX-Laboratory for Modeling and Scientific Computing, Dipartimento di Matematica, Politecnico di Milano, Piazza Leondardo da Vinci 32, 20133 Milano, Italy.
\and
P. F. Antonietti \at\email{paola.antonietti@polimi.it}
\and
M. Sarti \at\email{marco.sarti@polimi.it} 
\and
M. Verani \at\email{marco.verani@polimi.it} 
\and 
P. Houston \at School of Mathematical Sciences, University of Nottingham, University Park, Nottingham, NG7 2RD, United Kingdom.\\\email{paul.houston@nottingham.ac.uk}
\and 
X.~Hu \at Department of Mathematics, Tufts University, 503 Boston Avenue Bromfield-Pearson, Medford, MA 02155 \\
\email{Xiaozhe.Hu@tufts.edu}
}

\begin{document}

\maketitle
\begin{abstract}
In this paper we analyze the convergence properties of two-level and W-cycle multigrid solvers for the numerical
solution of the linear system of equations arising from $hp$-version symmetric interior penalty discontinuous Galerkin discretizations 
of second-order elliptic partial differential equations on polygonal/polyhedral meshes. We prove that the two-level method converges uniformly with respect to the granularity of the grid and the 
polynomial approximation degree $p$, provided that the number of smoothing steps, which depends on $p$, is chosen sufficiently 
large. An analogous result is obtained for the W-cycle multigrid algorithm, which is proved to be uniformly convergent with 
respect to the mesh size, the polynomial approximation degree, and the number of levels, provided the 
number of smoothing steps is chosen sufficiently large. Numerical experiments are presented which underpin the 
theoretical predictions; moreover, the proposed multilevel solvers are shown to be convergent in practice, even when 
some of the theoretical assumptions are not fully satisfied.
 
\end{abstract}
\section{Introduction}

The original articles concerned with the construction and mathematical analysis
of Discontinuous Galerkin (DG) methods date back over 50 years ago. 
For hyperbolic partial differential equations, in 1973 Reed \& Hill, cf. \cite{ReedHill73}, developed 
the first DG discretization of the neutron transport equation. Independently, DG methods were constructed for elliptic
problems based on weakly
enforcing Dirichlet boundary conditions; see, for example, \cite{Aubin1970,Babuska73,Lions68,nitsche}.
In particular, we highlight the works of Nitsche \cite{nitsche} and Baker \cite{Baker77},  which form the basis
of the class of interior penalty DG methods, cf. also \cite{Arn82,Whee78}.
Since the very early work, DG methods were partially abandoned, in part due to the increase in the number of degrees of freedom compared, for instance, with their conforming counterparts. However, in the last two decades there has been a renewed interest in the field of discontinuous discretizations both from a theoretical and computational viewpoint, cf. \cite{CockKarnShu00,HestWar,Rivi08,DiPiErn12},
for example. This resurgence is due to the inherent advantages offered by DG schemes, such as, for example, the limited interelement communication, which is restricted only to neighbouring elements, the local conservativity property, the simplicity in treating meshes with hanging nodes, and the development of efficient $hp$-adaptivity refinement strategies. Moreover, recently in 
\cite{BasBotColSub,Basetal12,BasBoetal12,Canetal14} it has been shown that the underlying DG polynomial bases may be efficiently constructed in the physical frame, without needing to map local polynomial spaces defined in a given reference/canonical frame. In this way, DG methods can easily deal with general-shaped elements, including polygonal/polyhedral elements, cf. 
\cite{AntBreMar09,dg_cfes_2012,Canetal14,BasBotColSub,CangianiDongGeorgoulisHouston_2016,AntFacRusVer2016,CangianiDongGeorgoulis_2016,CDGH_book} \added{and the recent review paper \cite{Anetal16_review}}. 
The flexibility of DG methods in handling general meshes has no immediate
counterpart in the conforming framework, where the design of suitable finite element spaces for meshes of polygons/polyhedra is far from being a trivial task. Several examples include the Composite Finite Element Method \cite{HackSau97,HackSau972}, the Polygonal Finite Element Method \cite{SukTab04,TabSuk08}, the Extended Finite Element Method \cite{FriBel10}, \added{the Mimetic Finite Difference Method \cite{HymShaStei97,BreLipSim05,BreLipSha05,BreBufSha06,BeiManLip_2014,Anetal16}} and the most recent Virtual Element Method \cite{Beietal13,BeiBreMarRus16,BeiBreMarRus16b,AntBeiMorVer,AntBeiScaVer16}.

At present, the design of solvers and preconditioners for DG discretizations on nonstandard grids lends itself to huge developments in the field of numerical analysis. Indeed, to the best of our knowledge, the only study regarding solution techniques for this class of problems is reported in \cite{AntGiaHou13}, where a nonoverlapping Schwarz preconditioner for composite DG finite element
methods on complicated domains is analyzed, \added{see also the recent paper \cite{AntHouSme} where optimal bounds for nonoverlapping Schwarz preconditioners for $hp$-version DG methods on standard shape-regular grids have been obtained.} In the current article we exploit the theoretical framework developed in \cite{Canetal14} to study the performance of a two-level and W-cycle multigrid solver.  The possibility to employ general-shaped elements in the physical framework makes the choice of multilevel schemes natural. The flexibility afforded by this approach allows us to define the set of grids needed in the
multigrid algorithm by agglomeration; thereby, the definition of the associated subspaces is straightforward, since inter-element continuity is not required. This property overcomes the usual difficulties encountered in the construction of agglomeration multigrid schemes in the conforming framework, where the agglomeration strategy must be followed by a proper definition of the conforming subspaces. In \cite{Chanetal98}, for example, the sublevels are obtained by combining a graph based agglomeration algorithm and re-triangulations, thus resulting in a set of non-nested grids, while the associated nested subspaces are defined by 
introducing suitable interpolation operators. The resulting V-cycle multigrid algorithm converges uniformly with respect to 
the meshsize $h$ provided that the number of levels is kept fixed.

In this paper we analyze the convergence of a two-level scheme and W-cycle multigrid method for the solution of the linear system of equations arising from the $hp$-version of the interior penalty DG scheme on polygonal/polyhedral meshes \cite{Canetal14},
thereby, extending the theoretical framework developed in \cite{AntSarVer15} for standard quasi-uniform \added{triangular/quadrilateral} meshes, \added{cf. also \cite{AntSarVer16} for three-dimensional numerical experiments}. Our analysis is based on the smoothing and approximation properties associated with the proposed method: the former corresponds to a Richardson iteration, whose study requires a result concerning the spectral properties of the stiffness matrix, while  the latter is inherent to the interior penalty DG scheme itself and exploits the error estimates derived in \cite{Canetal14}. 
\textcolor{black}{We show that, under suitable assumptions on the agglomerated coarse grid, both the two-level and the W-cycle multigrid schemes converge uniformly with respect to the granularity of the underlying partition and the polynomial approximation degree $p$, provided that the number of smoothing steps is chosen of order \textcolor{black}{$p^{2+\mu}$, with $\mu=0,1$}. 
Throughout the analysis, we also track the dependence of the error reduction factor of the two solvers on the geometric properties of the agglomerated grids, thereby recovering a similar result to the case when standard quasi-uniform \added{triangular and/or quadrilateral} meshes are employed.}

The rest of this paper is organized as follows. In Section~\ref{sec:theory} we introduce the interior penalty DG scheme
for the discretization of second-order elliptic problems on general meshes consisting of polygonal/polyhedral elements. Then
in Section~\ref{sec:preliminary}, we recall some preliminary analytical results concerning this class of schemes. 
In Section~\ref{sec:mulgr} we define the multilevel framework and introduce several technical results. We then focus first on the analysis of \added{the} two-level method, followed by the extension to the W-cycle multigrid solver. The main theoretical results are investigated through a series of numerical experiments presented in Section~\ref{sec:numerical}, \added{where we also present a comparison with an unsmoothed Algebraic Multigrid method.}
%In particular, we show that, in general, the limitation on the number of levels 
%employed in the W-cycle multigrid solver does not seem to be restrictive in practice.
\added{Finally, in Section~\ref{sec:conclusions} we draw some conclusions.}

%%%%%%%%%%
\section{Model problem and discretization}
\label{sec:theory}
We consider the weak formulation of the Poisson problem, subject to a homogeneous Dirichlet boundary condition: find $u\in V = H^2(\Omega)\cap H_0^1(\Omega)$ such that
\begin{equation}
\int_\Omega \nabla u \cdot \nabla v\ dx =\int_\Omega f v\ dx \, \qquad \forall v\in V,
\label{weak}
\end{equation}
with $\Omega\in \mathbb{R}^d$, $d = 2,3$, a convex polygonal/polyhedral domain with Lipschitz boundary and $f$ a given function in $L^2(\Omega)$.

For the sake of brevity, throughout this article, we write $x\lesssim y$ and $x\gtrsim y$ in lieu of $x\leq C y$ and $x\geq C y$, respectively, for a positive constant $C$ independent of the discretization parameters. Moreover, $x\approx y $ means that there exist constants $C_1, C_2>0$ such that $C_1 y \leq x \leq C_2y$. When required, the constants will be written explicitly.

In view of the forthcoming multigrid analysis, we denote by $\{\mcal[T][j]\}_{j=1}^J$ a sequence of partitions of the domain \Om, each of which consists  of  disjoint open polygonal/polyhedral elements \elem of diameter \h{\elem}, such that $\overline\Om = \bigcup_{\elem\in\mcal[T][j]}\bar\elem$, $j=1,\dots,J$.  We denote the mesh size of \mcal[T][j], $j=1,\dots,J$,
by $\h{j} =\max_{\elem\in\mcal[T][j]}\h{\elem}$. To each \mcal[T][j], $j=1,\dots,J$, we associate the corresponding \paul{DG} finite element space $V_j$, $j=1,\dots,J$, defined as
\begin{equation}
\label{Vj}
V_j =\{v\in L^2(\Om):v|_\elem\in \mcal[P][p_j](\elem),\elem\in\mcal[T][j]\},
\end{equation}
where $\mcal[P][p_j](\elem)$ denotes the space of polynomials of total degree at most $p_j\geq1$ on $\elem\in\mcal[T][j]$. 
A suitable choice of the sequences $\{\mcal[T][j]\}_{j=1}^J$ and $\{V_j\}_{j=1}^J$ leads to the so-called 
$h$- \added{and} $p$-multigrid schemes. In particular, the $h-$multigrid method is based on employing
a constant polynomial approximation degree for each $j$, $j=1,\dots,J$, (\emph{i.e.}, $p_j=p$), 
on a set of nested partitions $\{\mcal[T][j]\}_{j=1}^J$, such that the coarse level \mcal[T][j-1], $j=2,\dots,J$, is obtained by agglomeration from \mcal[T][j] in such a way that
\begin{alignat}{3}
\label{hjhj1}
\h{j-1}&\lesssim\h{j}&&\leq \h{j-1}\qquad \forall j = 2,\dots, J,
\end{alignat} 
\emph{i.e.}, we \paul{assume} a bounded variation hypothesis between subsequent levels. In the $p$-multigrid method, the partition is kept fixed for any $j$, $j=1,\dots,J$, 
while we assume that the polynomial degrees vary moderately from one level to another, \emph{i.e.},
\begin{equation}
\label{pjpj1}
p_{j-1}\leq p_{j}\lesssim p_{j-1}\qquad \forall j = 2,\dots, J.
\end{equation}
%The $hp$-multigrid method is obtained by combining these two strategies. 
Note that with the above choices we obtain nested 
finite element spaces $V_j$, $j=1,\dots,J$, \emph{i.e.}, $V_1\subseteq V_2\subseteq\dots\subseteq V_J$.

\subsection{Grid assumptions}
In this section, we outline some key definitions and assumptions. For any $\mcal[T][j]$, $j=1,\dots,J$, when no hanging nodes/edges are included in the partition, we define the \emph{interfaces} of the mesh \mcal[T][j] as the set of $(d-1)$-dimensional facets of the elements $\elem\in\mcal[T][j]$. The presence of hanging nodes/edges, on the other hand, can be handled by defining the interfaces of $\mcal[T][j]$ as the intersection of the $(d-1)$-dimensional facets of neighboring elements. This implies that, for $d=2$, an interface will always consist of a \paul{piecewise linear} line segment, \paul{i.e., they consist of a set of $(d-1)$--dimensional simplices. However, in general for $d=3$, the interfaces of \mcal[T][j] will consist of general polygonal surfaces. Thereby, we assume that each planar section of each interface of an element $\elem\in\mcal[T][j]$ may be subdivided into a set of co-planar triangles ($(d-1)$--dimensional simplices). We refer to these $(d-1)$--dimensional simplices, whose union form the interfaces of $\mcal[T][j]$, as faces. With this notation, we assume that the sub-tessalation of element interfaces into $(d-1)$--dimensional simplices is given. We denote by $\mcal[F][j]$ the set of all mesh faces;} moreover, we have that $\mcal[F][j] = \mcal[F][j]^I\cup\mcal[F][j]^B$, where $\mcal[F][j]^I$ is the set of interior element faces of \mcal[T][j], such that $F \subseteq \partial\elem^+\cap\partial\elem^-$ for any $F\in\mcal[F][j]^I$, where $\elem^\pm$ are two adjacent elements in $\mcal[T][j]$.  The set $\mcal[F][j]^B$ contains the boundary element faces, \emph{i.e.}, $F\subset\partial\Om$ for $F\in \mcal[F][j]^B$.\\

\textcolor{black}{
We are now ready to introduce the following assumptions on the partitions $\mcal[T][j]$, $j=1,\dots,J$; cf.} \added{\cite{CangianiDongGeorgoulis_2016}}. 
\textcolor{black}{In the case of the $h$-multigrid scheme, 
these assumptions must be satisfied for the meshes generated by the underlying agglomeration process.
%%%%%%%%%%
\begin{assumption}
\label{hyp:grid_new}
Given $\elem \in \mcal[T][j]$, there exists a set of nonoverlapping  (not necessarily shape-regular) $d$--dimensional simplices $T_{\ell}\subseteq\elem$, \paul{$\ell=1,2,\ldots,n_\kappa$}, such that, for any face $F \subset \partial \elem$, $\overline{F}= \partial \overline{\elem}\cap \partial \overline{T_{\ell}}$, for some $\ell$,
\begin{equation}\label{eq:subcovering}
\begin{aligned}
&&\cup_{\ell=1}^{n_\kappa}\overline{T}_{\ell} \added{\subseteq} \overline{\kappa},
\end{aligned}
\end{equation}
and the diameter $h_{\elem}$ of $\elem$ can be bounded by
\begin{equation*}
\begin{aligned}
h_{\elem} \lesssim \frac{d \abs{T_{\ell}}}{\abs{F}},
&& \ell=1,2, \ldots,n_\kappa.
\end{aligned}
\end{equation*}
\end{assumption}
%%%%
\begin{remark}
\paul{
We point out that Assumption~\ref{hyp:grid_new} does not put a restriction 
on either the number of faces that an element possesses, or indeed the measure of a
face of an element $\elem \in \mcal[T][j]$, relative to the measure of the
element itself. This will be particularly important in the agglomeration 
procedure employed within our $h$-multigrid method, since as the number of 
levels increases, the number of faces that the resulting agglomerated elements
may contain grows, while 
their measure, relative to the element measure, may degenerate.}
\end{remark}
%%%%
\begin{remark}
As pointed out in \cite{CangianiDongGeorgoulis_2016}, meshes obtained by agglomeration of a finite number of polygons that are uniformly star-shaped with respect to the largest inscribed ball will automatically satisfy Assumption~\ref{hyp:grid_new}.  Therefore, from the practical point of view, given a fine-level mesh $\mcal[T][J]$ consisting of uniformly star-shaped elements, a finite number of agglomeration steps will produce a sequence of admissible grids. To allow the number of agglomeration levels to increase arbitrarily 
one can either \emph{i)} ensure that each of the agglomerated meshes satisfy Assumption~\ref{hyp:grid_new}; \emph{ii)} check, at each level, that the (slightly more restrictive) shape-regularity criterion on the agglomerates is satisfied.
\end{remark}
}
%%%%%%%%%%%%%%
\begin{assumption}
\label{hyp:shapereg}
For any $\elem\in\mcal[T][j]$, $j=1,\dots,J$, we assume that $h_\elem^d\geq|\elem|\gtrsim h_\elem^d,$
with $d=2,3$.
\end{assumption}
%The last key ingredient needed for our analysis is the following inverse estimate;
%
We next introduce the following additional mesh condition, cf.~\cite{CangianiDongGeorgoulisHouston_2016},
\added{which will be required in order to obtain the inverse estimates presented in Lemma \ref{lem:inverse}}.
\begin{assumption}\label{hyp:shape-regular-element}
\textcolor{black}{
Every polytopic element  $\kappa\in \mcal[T][j]$, 
admits a sub-tri\-ang\-ulation into at most $m_\kappa$ 
shape-regular simplices $\frak{s}_i$, $i=1,2,\dots , m_\kappa$, such that
$\bar{\kappa} = \cup_{i=1}^{\added{m_\kappa}} \bar{\frak{s}}_i$ and
$$
|\frak{s}_i| \added{\gtrsim} |\kappa|
$$
for all $i=1,\dots, m_\kappa$, for some $m_\kappa\in\mathbb{N}$.
\added{The hidden constant is independent of $\kappa$ and $\mcal[T][j]$}.
}
%$\hat{c}>0$, independent of $\kappa$ and $\mcal[T][j]$.
\end{assumption}
%%%%%%%%%
In view of the approximation result that will be presented in the next section we also introduce the following additional assumption.
\begin{assumption} \label{hyp:covering} Let $\mcal[T][j][\sharp]=\{\mcal[K]\}$, denote a covering of \Om consisting of shape-regular $d$-dimensional simplices $\mcal[K]$. We assume that, for any $\elem\in \mcal[T][j]$, there exists  $\mcal[K]\in\mcal[T][j][\sharp]$ such that $\elem\subset \mcal[K]$ and
$$\paul{\max_{\kappa\in {\cal T}_j} \,}
\textnormal{card} \left\{ \kappa^\prime \in {\cal T}_j : \kappa^\prime \cap \mathcal{K} \neq \emptyset, ~\mathcal{K}\in\mathcal{T}_j^{\sharp} ~\mbox{ such that } ~\kappa\subset\mathcal{K} \right\}\lesssim 1.$$
Consequently, for each pair $\kappa$, $\mathcal{K}\in\mathcal{T}_j^{\sharp}$, with $\kappa\subset\mathcal{K}$,
$\textnormal{diam}(\mcal[K])\lesssim \h{\elem}.$
\end{assumption}
%%%%%
\textcolor{black}{We also need the following assumption on the quality of agglomerated grids. 
\begin{assumption} \label{ass:agglomeration_new} 
For any $F\in\mcal[F][j]\cap\mcal[F][j-1]$, $j=2,\ldots,J$, we denote by $\elem^\pm_{j}$ and $\elem^\pm_{j-1}$ the neighboring elements sharing the face $F$ in \mcal[T][j] and \mcal[T][j-1], respectively. We assume that there exists $\Theta>0$ such that
\begin{equation*}
\begin{aligned}
& 1<\frac{h_{\elem_{j-1}^\pm}}{h_{\elem_j^\pm}}\leq \Theta\quad\forall F\in\mcal[F][j]\cap\mcal[F][j-1].
\end{aligned}
\end{equation*} 
\end{assumption}
We remark that Assumption~\ref{ass:agglomeration_new} is satisfied 
if the agglomeration algorithm preserves the shape-regularity of the elements.
In Figure~\ref{fig:agglo}, we show two examples of possible macroelements: the agglomerate on the left is not suitable 
to guarantee Assumption~\ref{ass:agglomeration_new} due to the presence of a dominant dimension, while the element on the right 
can be considered appropriate. Moreover, we note that the fulfilment of Assumption~\ref{ass:agglomeration_new} can be considered a good criterion in evaluating the quality of the agglomerated grids employed in the multigrid algorithm,
cf. Figure \ref{fig:agglo} for an illustration.
%%%%%%%%%%%%%%
\begin{figure}[t!]
\center
\includegraphics[scale=0.8]{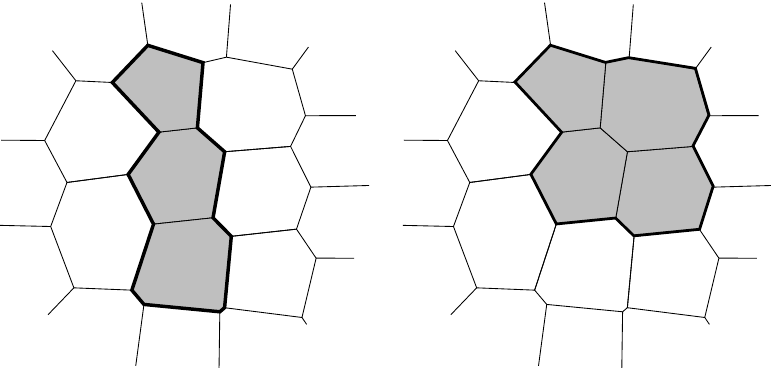}
\caption{\added{Examples of agglomerated elements. The agglomerated element on the left violates Assumption 
\ref{ass:agglomeration_new} whereas the one on the right satisfies Assumption~\ref{ass:agglomeration_new}}.}
\label{fig:agglo}
\end{figure}
%%%%%%%%%%%%%%%
%%%
}

\textcolor{black}{
Finally, to keep the notation as simple as possible, in the forthcoming analysis we will assume that,
for any $j=1,\dots,J$, the decompositions  $\mcal[T][j]$ are quasi-uniform, i.e., $\h{j} \approx \min_{\elem\in\mcal[T][j]}\h{\elem}$. We remark that the above assumption can be weakened and only a local bounded variation property is needed for our theoretical analysis; see Remark \ref{rem:theta} below for details.
}

%%%%%%%%%%%%
\subsection{DG formulation}
The definition of the proceeding DG method is based on employing suitable jump and average operators. 
To this end, for (sufficiently smooth) vector- and scalar-valued functions $\boldsymbol{\tau}$ and $v$, respectively, 
we define jumps and averages across $F\in\mcal[F][j]$, $j=1,\ldots,J$, as follows:
\begin{alignat*}{4}
\jump{\boldsymbol{\tau}} &= \boldsymbol{\tau}^+ \cdot \mathbf{n}^+ + \boldsymbol{\tau}^- \cdot \mathbf{n}^-,\quad& \average{\boldsymbol{\tau}} &= \frac{\boldsymbol{\tau}^+  +\boldsymbol{\tau}^-}{2},\qquad & F\in \mcal[F][j]^I,\\
\jump{v} &= v^+ \mathbf{n}^+ + v^-\mathbf{n}^-,& \average{v} &= \frac{v^+  +v^-}{2},\qquad & F\in \mcal[F][j]^I,\\
\jump{v} &= v^+\mathbf{n}^+,& \average{ \boldsymbol{\tau}} &=\boldsymbol{\tau}^+,&F\in \mcal[F][j]^B,
\end{alignat*}
where $v^\pm$ and $\boldsymbol{\tau}^\pm$ denote the traces of $v$ and $\boldsymbol{\tau}$ on $F$ taken from the interior of 
$\elem^\pm$, respectively, and $\mathbf{n}^\pm$ the outward unit normal vectors to $\partial\elem^\pm$, respectively, cf. \cite{Arnetal01}.
On any level $j$, $j=1,\dots,J$, we consider the bilinear form $\Aa[j]: V_j\times V_j\rightarrow \mathbb{R}$, 
corresponding to the symmetric interior penalty DG method, defined by
\begin{align}
\label{bilinearj}
\Aa[j][u][v] =& \sum_{\elem\in\mcal[T][j]}\int_\elem \nabla u\cdot\nabla v\ dx
-\sum_{F\in\mcal[F][j]}\int_F \left(\average{\nabla u}\cdot\jump{v}+\jump{u}\cdot\average{\nabla v}\right) \ ds\notag\\
&+\sum_{F\in\mcal[F][j]}\int_F\sigma_j\jump{u}\cdot\jump{v}\ ds,
\end{align}
where $\sigma_j\in L^\infty(\mcal[F][j])$ denotes the interior penalty stabilization function
\textcolor{black}{
$\sigma_j: \mcal[F][j]\rightarrow \mathbb{R}^+$, which is defined by
\begin{equation}\label{eq:stabilization_function}
\sigma_j(x)=
\left\{
\begin{aligned}
&C_{\sigma}^j \max_{\elem \in \{\elem^+,\elem^-\} }
\Big\{ \frac{p_j^2}{h_\elem}\Big\},\ 
&x\in F, ~F\in\mcal[F][j]^I,\ F \subset \partial\elem^+\cap\partial\elem^-,\\
&C_{\sigma}^j \frac{p_j^2}{h_{\elem}},\ &
x\in F, ~F\in\mcal[F][j]^B,\ F \subset \partial\elem^+\cap\partial\Omega,\\
\end{aligned}
\right.
\end{equation}
with $C_{\sigma}^j>0$ independent of $p_j$, $|F|$ and $|\elem|$. }

In this article, we develop two-level and W-cycle multigrid schemes to compute
the solution of the following problem on the finest level $J$: find $u_J\in V_J$ such that
\begin{equation}
\label{DGfem}
\Aa[J][u_J][v_J] = \int_\Om fv_J\ dx\quad\forall v_J\in V_J.
\end{equation}

%%%%%%%%%%
\section{Preliminary results}
\label{sec:preliminary}
\textcolor{black}{
We first recall the following trace-inverse inequality 
for polygonal/polyhedral elements.
\begin{lemma} \label{lem:inversepoly}
Assume that the sequence of meshes $\mcal[T][j]$, $j=1,\dots,J$, satisfies 
Assumption~\ref{hyp:grid_new}. Let $\elem\in\mcal[T][j]$, $j=1,\dots,J$, 
be a polygonal/polyhedral element, then \paul{the following bound} holds 
\begin{equation}
\begin{aligned}
\label{inversepoly}
\normL[v][2][\partial\kappa][2]\leq 
\CINVnew{j} \frac{p_j^2}{h_{\elem}}\normL[v][2][\elem][2]
&& \forall\, v\in\mcal[P][p_j](\elem),
\end{aligned}
\end{equation}
where $\CINVnew{j}$ is independent of $|\elem|$, $p_j$ and $v$.
\end{lemma}
The proof can be obtained with trivial modifications with respect to the ones given in \cite{CangianiDongGeorgoulis_2016,CDGH_book}.
For the sake of completeness we report it and refer to \cite{CangianiDongGeorgoulis_2016,CDGH_book} for further details.
\begin{proof}
From Assumption~\ref{hyp:grid_new}, there exists a set of nonoverlapping (not necessarily 
\paul{shape-regular}) $d$-simplicial elements $T_{\ell}\subseteq\elem$ 
\paul{such that, given a face $F\subset \partial \elem$, for some $\ell$, $1\leq \ell\leq n_\kappa$,  $\overline{F}=\partial \overline{\elem} \cap \partial \overline{T_{\ell}}$.}
Therefore, 
\paul{
\begin{equation*}
\normL[u][2][\partial\kappa][2]
= \sum_{F\subset\partial\kappa} \normL[u][2][F][2]
\lesssim p_j^2 \sum_{\ell=1}^{n_\kappa} \frac{\abs{F}}{\abs{T_\ell}}\normL[u][2][T_{\ell}][2]
\lesssim \frac{p_j^2}{h_\kappa} \sum_{\ell=1}^{n_\kappa} \normL[u][2][T_{\ell}][2]
\added{\leq} \frac{p_j^2}{h_\kappa} \normL[u][2][\kappa][2],
\end{equation*}
as required. Here, in the first inequality we have employed the following classical trace-inverse estimate on $d$-simplicial elements
\begin{equation*}
\normL[u][2][F][2]
\lesssim p_j^2\frac{\abs{F}}{\abs{T_\ell}}\normL[u][2][T_{\ell}][2],
\end{equation*}
cf. \cite{Sch98,Geo08}, for example; the second bound exploits
Assumption~\ref{hyp:grid_new}, namely, that 
$\abs{F}{\abs{T_\ell}}^{-1} \lesssim dh^{-1}_{\elem}$.}
\end{proof}
}

Next, we endow the finite element spaces $V_j$, $j=1,\dots,J$, with the following DG norm:
\begin{equation}
\normDG[w][j][2]=\sum_{\elem\in\mcal[T][j]}\int_\elem |\nabla w|^2\ dx + \sum_{F\in\mcal[F][j]} \int_F \sigma_j|\jump{w}|^2\ ds.
\end{equation}
The well--posedness of the DG formulation is established in the following lemma
\begin{lemma}
\label{lem:contcoerc}
\added{Assume that the sequence of meshes $\mcal[T][j]$, $j=1,\dots,J$, satisfies Assumption~\ref{hyp:grid_new} and that the constant $C_\sigma^j$, $j=1,\dots,J$, appearing in the definition \eqref{eq:stabilization_function} of the stabilization function is chosen sufficiently large.} Then, the following continuity and coercivity bounds, respectively, hold
\begin{alignat}{3}
\Aa[j][u][v]&\leq \Ccont  \normDG[u][j]\normDG[v][j]\quad &&\forall u,v\in V_j, \label{cont}\\
\Aa[j][u][u]&\geq \Ccoer\normDG[u][j][2]\quad &&\forall u\in V_j,\label{coerc}
\end{alignat}
where $\Ccont$ and $\Ccoer$ are positive constants, independent of the discretization parameters.
\end{lemma}\par

The proceeding error estimates are based on the following approximation result, which is a simplified version of the 
analogous bound presented in \cite[Proof of Theorem~5.2]{Canetal14}. To this end, we define 
$\mathcal{E}:H^s(\Omega) \rightarrow H^s({\mathbb R}^d)$, $s \in {\mathbb N}_0$, such that
$\mathcal{E}v|_{\Omega}=v$, to denote the extension operator
presented in Stein \cite{stein}.

\begin{lemma}
\label{lem:interpDG}
\textcolor{black}{
Assume that Assumptions~\ref{hyp:grid_new} and \ref{hyp:covering} hold. Let 
$v|_\elem\in H^{k}(\elem)$, $k>d/2$, such that $\mathcal{E}v|_{\mathcal{K}}\in H^{k}(\mathcal{K})$,
for each $\elem\in\mcal[T][j]$, $j=1,\dots,J$, where $\elem\subset\mathcal{K}$, $\mathcal{K}\in\mathcal{T}_j^{\sharp}$. 
Then there exists a projection operator $\widetilde{\Pi}_j: L^2(\Omega)\rightarrow V_j$ such that
\begin{align}
\label{interpDG}
\normDG[v-\tilde{\Pi}_jv][j]&\leq \Cinterpnew{j}\frac{h_j^{s-1}}{p_j^{k\textcolor{black}{-1-\mu/2}}}\normH[v][k][\Om],
\end{align}
where $s = \min\{p_j+1,k\}$, and the constant $\Cinterpnew{j}$ depends on the shape-regularity 
constant of the covering $\mathcal{T}_j^{\sharp}$, but is independent of the discretization parameters, 
as well as the number of faces per element and the relative measure of the faces. \textcolor{black}{Here, $\mu=0$ whenever a $p-$optimal interpolant can be constructed  and $\mu=1$ otherwise.}}
\end{lemma}
\par
Next, we state error bounds for the underlying interior penalty DG scheme in terms of both the DG and $L^2(\Omega)$-norms.
\begin{theorem}
\label{thm:errors}
\textcolor{black}{
Assume that Assumptions~\ref{hyp:grid_new}  and \ref{hyp:covering} hold. Denote by 
$u_j\in V_j$, $j=1,\ldots,J$, the DG solution of problem \eqref{DGfem} posed on level $j$, \emph{i.e.}, 
$$\Aa[j][u_j][v_j] = \int_\Om fv_j\ dx\quad\forall v_j\in V_j.$$ 
If the solution $u$ of \eqref{weak} satisfies $u|_\elem\in H^k(\elem)$, $k>1+d/2$, such that
$\mathcal{E}u|_{\mathcal{K}}\in H^{k}(\mathcal{K})$,
for each $\elem\in\mcal[T][j]$, $j=1,\dots,J$, where $\kappa\subset\mathcal{K}$, $\mathcal{K}\in\mathcal{T}_j^{\sharp}$,
then the following bounds hold
\begin{align}
\normDG[u-u_j][j]&\leq \CGnew{j}\frac{h_j^{s-1}}{p_j^{k\textcolor{black}{-1-\mu/2}}}\normH[u][k][\Om],\label{DGerror}\\
\normL[u-u_j][2][\Om]&\leq \CLtwonew{j}\frac{\h{j}^{s}}{p_j^{k\textcolor{black}{-\mu}}}\normH[u][k][\Om],\label{L2error}
\end{align}
where $s = \min\{p_j+1,k\}$ and the constants $\CGnew{j}$ and $\CLtwonew{j}$ are independent of the discretization parameters. \textcolor{black}{Here, $\mu=0$ whenever a $p-$optimal interpolant can be constructed  and $\mu=1$ otherwise.}
}
\end{theorem}
\added{Before proceeding with the proof, we point out that the above error bounds 
\paul{hold provided Assumptions~\ref{hyp:grid_new} and~\ref{hyp:covering} are satisfied;
however, we stress that no limitation is placed on the maximum number of 
faces that each polygonal/polyhedral element may possess. Moreover, there is no
restriction on the relative size of each face of an element compared to its diameter.}}
\begin{proof}
\textcolor{black}{
The error bound \eqref{DGerror} stems from the general result derived in 
\cite[Theorem~5.2]{Canetal14} \paul{under the condition that 
Assumptions~\ref{hyp:grid_new} and~\ref{hyp:covering} hold}. Thereby, we now
proceed with the proof of the bound on the $L^2(\Omega)$-norm of the error, cf. \eqref{L2error}. To this end, we
employ a standard duality argument: let $w\in V$, be the solution of the problem
\begin{equation*}
\paul{\Aa[j][v][w]} = \int_\Omega (u-u_j)v\ dx\qquad \forall v\in V,
\label{adjoint}
\end{equation*}
$j=1,\ldots,J$.
Exploiting a standard elliptic regularity assumption, we note that
\begin{equation*}
\normH[w][2]\lesssim \normL[u-u_j].
\label{ellipticw}
\end{equation*}
According to Galerkin orthogonality, we immediately obtain
\begin{align*}
\normL[u-u_j][2][\Om][2] &= \Aa[j][u-u_j][w]\\
&= \Aa[j][u-u_j][w-w_I]\\
&\lesssim \normDG[u-u_j][j]\normDG[w-w_I][j]
\end{align*}
for all $w_I\in V_j$. Hence, selecting $w_I=\tilde{\Pi}_j w$, employing 
\eqref{interpDG} gives
\begin{equation}
\normDG[w-w_I][j]\lesssim\Cinterpnew{j}\frac{h_j}{p_j^{1/2}}\normH[w][2][\Om]
\lesssim \Cinterpnew{j}\frac{h_j}{p_j^{\textcolor{black}{1-\mu/2}}}\normL[u-u_j][2][\Om],
\end{equation}
which together with \eqref{DGerror} gives the desired result.
}
\end{proof}\par
%%%%%%%%%%%%%%

\paul{
Equipped with Assumption~\ref{hyp:shape-regular-element}, we now quote the following result
from \cite{CangianiDongGeorgoulisHouston_2016}; for brevity the proof is omitted. 
However, we point out that the proof presented in \cite{CangianiDongGeorgoulisHouston_2016}
holds under slightly weaker mesh conditions; for simplicity of presentation, this level of
detail is omitted.
\begin{lemma}
\label{lem:inverse}
\textcolor{black}{Assume that Assumptions~\ref{hyp:shapereg}  and~\ref{hyp:shape-regular-element} hold}. Then, for any $v\in V_j$, $j=1,\ldots,J$, the following inverse estimate holds
\begin{equation}
\normL[\nabla u][2][\elem][2]\leq \CInew{j} p_j^4 h_\elem^{-2}\normL[u][2][\elem][2],
\end{equation}
where $\CInew{j}>0$ is independent of the discretization parameters.
\end{lemma}
}
%%%%%%%%%%%%%%%%%%%%%%
%\begin{proof}
%\textcolor{black}{
%We recall that, thanks to Assumption~\ref{hyp:grid_new},
%$\elem$ can be covered by a set of nonoverlapping (not necessarily star-shaped) $d$-simplicial elements $T_{\ell}\subseteq\elem$ 
%such that $\overline{F}=\partial \overline{\elem} \cap \partial \overline{T_{\ell}}$ for any face $F$ of $\elem$, and 
%\begin{equation*}
%\begin{aligned}
%h_{\elem} \lesssim \frac{d \abs{T_{\ell}}}{\abs{F}}.
%&&\ell=1, 2,\ldots,
%\end{aligned}
%\end{equation*}
%We then have, by using the following standard  inverse estimates on 
%simplicial elements
%\begin{equation*}
%\normL[\nabla u][2][T^{\ell}][2] \lesssim \frac{p_j^4}{h_{T_\ell}^{2}}\normL[u][2][T_{\ell}][2],
%\end{equation*}
%that
%\begin{align}
%\normL[\nabla u][2][\elem][2] = \sum_{\ell}
%\normL[\nabla u][2][T^{\ell}][2]\lesssim 
%\sum_{\ell}\frac{p_j^4}{h_{T_\ell}^{2}}\normL[u][2][T_{\ell}][2]
%\lesssim \frac{p_j^4}{h_{\elem}^2}\left(\max_{\ell=1, 2,\ldots} \frac{h_{\elem}^2}{h_{T_\ell}^{2}}
%\right)\normL[u][2][\elem][2].
%\end{align}
%The thesis follows using Assumption~\ref{hyp:grid_new} and observing that, for any $d-$dimensional simplicial element, its area $\abs{T_{\ell}}$ is given by  $\abs{T_{\ell}}=\abs{F}\abs{h_F^{\dag}}/d $, being $h_F^{\dag}$ its height with respect to the face $F$, and therefore 
%$\abs{T_{\ell}}\leq \abs{F} h_{T_{\ell}}/d$, since $h_F^{\dag}\leq h_{T_{\ell}}$ by definition, that is
%\begin{equation}
%\begin{aligned}
%& h_{\elem} \lesssim \frac{d \abs{T_{\ell}}}{\abs{F}}
%\leq h_{T_{\ell}}
%&& \forall \ell=1, 2, \dots.
%\end{aligned}
%\end{equation}
%}
%\end{proof}
The inverse estimate presented in Lemma~\ref{lem:inverse} is fundamental to the proof of the 
following upper bound on the maximum eigenvalue of \Aa. We recall that the analogous result on standard grids can be found in \cite{AntHou}, \added{cf. also \cite{AntHou13}}.
\begin{theorem}
\label{thm:eig}
\textcolor{black}{
Given that Assumptions~\ref{hyp:grid_new}, \ref{hyp:shapereg}, 
\paul{and~\ref{hyp:shape-regular-element}} hold, then for any 
$u\in V_j$, $j=1,\ldots,J$, we have that
\begin{equation}
\label{eig}
\Aa[j][u][u]\leq \Ceig{j}\frac{p_j^4}{\hmin{j}^2} \normL[u][2][\Om][2],
\end{equation}
where the constant $\Ceig{j}$ is independent of the discretization parameters.
}
\end{theorem}
\begin{proof}
\textcolor{black}{
Given the continuity of the bilinear forms \Aa[j] stated in Lemma~\ref{lem:contcoerc}, we restrict ourselves to estimate the 
two terms involved in the DG norm. The local contributions of the $H^1$ seminorm can be simply bounded by applying 
Lemma~\ref{lem:inverse} and the quasi-uniformity of the partition, \emph{i.e.}, 
\begin{equation}
\label{DGinv}
\sum_{\elem\in\mcal[T][j]}|u|_{H^1(\elem)}^2
\leq\sum_{\elem\in\mcal[T][j]} \CInew{j} p_j^4 h_\elem^{-2}\normL[u][2][\elem][2]\leq 
\textcolor{black}{\left(\max_{\elem\in\mcal[T][j]} \CInew{j}\right)}\frac{p_j^4}{h_j^{2}}\normL[u][2][\Om][2].
\end{equation}
\paul{
To bound the norm of the jump across $F\in\mcal[F][j]$, we employ the 
inverse inequality \eqref{inversepoly}; thereby, we get
\begin{align}
\sum_{F\in \mcal[F][j]} \normL[\sigma_j^{1/2}\jump{u}][2][F][2]
\lesssim& ~\sum_{\kappa\in \mcal[T][j]} 
	\normL[\sigma_j^{1/2}\jump{u}][2][\partial\kappa][2] 
%\lesssim ~C^j_\sigma \CINVnew{j} 
%	\sum_{\kappa\in \mcal[T][j]} \frac{p_j^4}{h_\kappa^2} \normL[u][2][\kappa][2] \\
\lesssim ~ C^j_\sigma \CINVnew{j} \frac{p_j^4}{h_j^2}  \normL[u][2][\Omega][2].
\end{align}
}
The statement of the theorem immediately follows based on summing the above bounds.
}
\end{proof}\par

The theoretical results derived in this section form the basis of the analysis of the proposed multigrid 
algorithms presented in the following section.

%%%%%%%%%%%%

%%%%%%%%%%
\section{Two-level and W-cycle multigrid algorithms} 
\label{sec:mulgr}

The forthcoming analysis is based on the classical multigrid theoretical framework already employed in \cite{AntSarVer15} 
for high-order DG schemes on standard quasi-uniform meshes. The two key ingredients in the construction of our 
proposed multigrid schemes are the inter-grid transfer operators and the smoothing scheme. 
The prolongation operator connecting the space $V_{j-1}$ to $V_{j}$, $j=2,\ldots,J$,
 is denoted by $I^{j}_{j-1}:V_{j-1}\rightarrow V_{j}$, while its adjoint with respect to the $L^2(\Omega)$-inner product 
 $(\cdot,\cdot)$ is the restriction operator $I^{j-1}_{j}:V_{j}\rightarrow V_{j-1}$ 
 \paul{defined by}
$$
(I^{j}_{j-1}v,w)  = (v,I^{j-1}_{j}w)\qquad \forall v\in V_{j-1},w\in V_{j}.
$$
As a smoothing scheme, we choose a Richardson iteration, whose operator is defined as:
\begin{equation}
\label{defsmooth}
B_j = \Lambda_j\textnormal{Id}_j,
\end{equation}
with $\textnormal{Id}_j$ the identity operator on level $V_j$, and $\Lambda_j\in\mathbb{R}$ is an upper bound 
for the spectral radius of the operator $A_j:V_j\rightarrow V_j$, defined by
\begin{equation}
\label{defAj}
(A_ju,v)=\Aa[j][u][v]\quad\forall u,v\in V_j, ~~j=1,\ldots,J.
\end{equation}
For the definition of the solvers, we first address the two-level method. Given the problem $A_J u_J = f_J$
with $A_J:V_J\rightarrow V_J$ defined according to \eqref{defAj}, and $f_J\in V_J$ such that 
$$(f_J,v) = \int_\Om f v\ dx\quad \forall v\in V_J,$$
in Algorithm~\ref{alg:2lvl}  we outline the two-level cycle, where $\mathsf{MG_{2lvl}} (z_0,m_1,m_2)$ denotes the approximate solution obtained after one iteration, with initial guess $z_0$ and $m_1$, $m_2$ pre- and post-smoothing steps, respectively.
%%%%%%%%%%%
\begin{algorithm}[t!]
\caption{Two-level scheme}
\label{alg:2lvl}
\begin{algorithmic}
\State \underline{{\it  Pre-smoothing}}:
\For{$i=1,\dots,m_1$} \State $z^{(i)}=z^{(i-1)}+B_J^{-1}(f_J-A_Jz^{(i-1)});$ \EndFor\vspace{0.3cm}
\State \underline{{\it  Coarse grid correction}}:
\State $r_{J-1} = I_{J}^{J-1}(f_J-A_Jz^{(m_1)})$;
\State $e_{J-1} = A_{J-1}^{-1}r_{J-1}$;
\State $z^{(m_1+1)}=z^{(m_1)}+I_{J-1}^{J}e_{J-1}$;\vspace{0.3cm}
\State \underline{{\it  Post-smoothing}}:
\For{$i=m_1+2,\dots,m_1+m_2+1$} \State $z^{(i)}=z^{(i-1)}+B_J^{-1}(f_J-A_Jz^{(i-1)});$ \EndFor\vspace{0.3cm}
\State $\mathsf{MG_{2lvl}} (z_0,m_1,m_2)=z^{(m_1+m_2+1)}.$
\end{algorithmic}
\end{algorithm}

As a multilevel extension of Algorithm~\ref{alg:2lvl}, we consider a standard W-cycle scheme. On level $j$, we consider $A_j z=g$,
for a given $g\in V_j$. The approximate solution obtained by applying the $j$-th level iteration to the above linear system, 
with initial guess $z_0$ and  $m_1$, $m_2$ pre- and post-smoothing steps, respectively, is 
denoted by $\mathsf{MG}_\mathcal{W} (j,g,z_0,m_1,m_2)$. On the coarsest level $j = 1$, the corresponding subproblem is solved  
based on employing a direct method, \emph{i.e.},
$$\mathsf{MG}_\mathcal{W} (1,g,z_0,m_1,m_2) = A_1^{-1}g,$$
while for $j>1$ we apply the recursive procedure outlined in Algorithm~\ref{alg:multilevel}.
We observe that Algorithm~\ref{alg:2lvl} can be considered as a special case of Algorithm~\ref{alg:multilevel}, 
corresponding to $J=2$.

\begin{algorithm}[t!]
\caption{Multigrid W-cycle scheme}
\label{alg:multilevel}
\begin{algorithmic}
\If{$j=1$}
\State \textcolor{black}{$\mathsf{MG}_\mathcal{W} (1,g,z_0,m_1,m_2)=A_1^{-1}g.$}
\Else
\State \underline{{\it  Pre-smoothing}}:
\For{$i=1,\dots,m_1$} \State $z^{(i)}=z^{(i-1)}+B_j^{-1}(g-A_jz^{(i-1)});$ \EndFor\vspace{0.3cm}
\State \underline{{\it  Coarse grid correction}}:
\State $\textcolor{black}{r_{j-1}} = I_j^{j-1}(g-A_jz^{(m_1)})$;
\State $\overline{e}_{j-1} = \mathsf{MG}_\mathcal{W} (j-1,r_{j-1},0,m_1,m_2)$;
\State $e_{j-1} = \mathsf{MG}_\mathcal{W} (j-1,r_{j-1},\overline{e}_{j-1},m_1,m_2)$;
\State $z^{(m_1+1)}=z^{(m_1)}+I_{j-1}^je_{j-1}$;\vspace{0.3cm}
\State \underline{{\it  Post-smoothing}}:
\For{$i=m_1+2,\dots,m_1+m_2+1$} \State $z^{(i)}=z^{(i-1)}+B_j^{-1}(g-A_jz^{(i-1)});$ \EndFor\vspace{0.3cm}
\State $\mathsf{MG}_\mathcal{W} (j,g,z_0,m_1,m_2)=z^{(m_1+m_2+1)}.$
\EndIf
\end{algorithmic}
\end{algorithm}

\subsection{Convergence analysis of the two-level method}
\label{sec:2lvlconv}
We first define the following norms based on the operator $A_j$, $j=1,\dots,J$,
\begin{equation}
\label{discrnorm}
\ltrivert v\rtrivert_{s,j}=\sqrt{(A_j^sv,v)_j}\qquad \forall s\in\textcolor{black}{\mathbb{N}\cup \left\{0\right\}},\ v\in V_j,\quad j=1,\dots,J.
\end{equation}
Hence, 
\begin{equation}
\label{norm1Aj}
\ltrivert v\rtrivert_{1,j}^2 = 
(A_jv,v)_j=\Aa[j][v][v],\quad \ltrivert v\rtrivert_{0,j}^2 =(v,v)_j=\|v\|_{L^2(\Omega)}^2\quad\forall v\in V_j.
\end{equation}

In order to undertake the convergence analysis of the two-level solver outlined in Algorithm~\ref{alg:2lvl}, we follow the 
approach developed in \cite{AntSarVer15}. We then provide an estimate based on the error propagation operator,
which is defined by
\begin{equation}
\label{E2lvl}
\mathbb{E}_{m_1,m_2}^{\mathsf{2lvl}}v = G_J^{m_2}(\textnormal{Id}_J-I_{J-1}^JP_J^{J-1})G_J^{m_1},
\end{equation}
with $G_J = \textnormal{Id}_J-B_J^{-1}A_J$, and the operator $P_{J}^{J-1}:V_{J}\rightarrow V_{J-1}$ defined as
\begin{equation}
\Aa[J-1][P_{J}^{J-1}v][w] = \Aa[J][v][I^{J}_{J-1}w]\qquad \forall v\in V_{J}, w\in V_{J-1}.
\label{project}
\end{equation}
We now study separately the {\it  smoothing property} and the {\it  approximation property}. 
We also point out that that by Theorem~\ref{thm:eig}, we can bound $\Lambda_j$, $j=1,\ldots,J$, 
in \eqref{defsmooth} as follows
\begin{equation}
\label{Lambda}
\Lambda_j\lesssim \Ceig{j}\frac{p_j^4}{\hmin{j}^2}.
\end{equation}
\par
The last result is employed to prove the smoothing property in the next lemma; see \cite[Lemma~4.3]{AntSarVer15} for the proof. 
%%%%%
\begin{lemma}[Smoothing property]
\label{lem:smooth}
\textcolor{black}{Given that Assumptions~\ref{hyp:grid_new}, \ref{hyp:shapereg}, 
and~\ref{hyp:shape-regular-element} hold, then} for any $v\in V_j$, $j=1,\ldots,J$, we have
\begin{equation}
\begin{aligned}
\ltrivert G_j^mv\rtrivert_{1,j}&\leq\ltrivert v\rtrivert_{1,j},\\
\ltrivert G_j^mv\rtrivert_{s,j} &\lesssim \Ceig{j}^{(s-t)/2}p_j^{2(s-t)}\hmin{j}^{t-s}(1+m)^{(t-s)/2}\ltrivert v \rtrivert_{t,j},
\label{smoothingprop}
\end{aligned}
\end{equation}
for $0\leq t< s\leq 2$ and $m\in\mathbb{N}\setminus\{0\}$.
\end{lemma}
The {\it approximation property} \paul{stems from} 
exploiting the $L^2(\Omega)$ error estimates stated
in \eqref{L2error} on levels $J$ and $J-1$.
%%%%%
\begin{lemma}[Approximation property]
\label{lem:approx}
\textcolor{black}{Assume that Assumptions~\ref{hyp:grid_new} and \ref{hyp:covering} hold.} 
\textcolor{black}{Let $\mu$ be defined as in Lemma \ref{lem:interpDG}.}
For any $v\in V_J$, the following inequality holds
\begin{equation}
\begin{aligned}
& \ltrivert (\textnormal{Id}_J-I_{J-1}^JP_J^{J-1})v\rtrivert_{0,J}\lesssim (\CLtwonew{J} + \CLtwonew{J-1})\frac{\h{J-1}^2}{p_{J-1}^{\textcolor{black}{2-\mu}}}\ltrivert v\rtrivert_{2,J}.
&&
\end{aligned}
\label{approxprop}
\end{equation}
\end{lemma}
\begin{proof}
For any $v\in V_J$, we consider the following equality
\begin{align}
\ltrivert (\textnormal{Id}_J-I_{J-1}^JP_J^{J-1})v\rtrivert_{0,J} &=\normL[(\textnormal{Id}_J-I_{J-1}^JP_J^{J-1})v] \\
&= ~\sup_{0\neq
\phi \in L^2(\Omega)} 
\frac{\int_{\Omega}\phi(\textnormal{Id}_J-I_{J-1}^JP_J^{J-1})v\ dx}{\normL[\phi]}.
\label{dualratio}
\end{align}
Next, we consider the solution $\eta$ of the following problem
\begin{equation*}
\begin{aligned}
\int_{\Omega} \nabla \eta\cdot\nabla v\ dx &= \int_{\Omega}\phi v\ dx &&\forall v\in V,
\end{aligned}
\end{equation*}
for $\phi\in L^2(\Omega)$, and let $\eta_J\in V_J$ and $\eta_{J-1}\in V_{J-1}$ be the corresponding DG approximations  
in $V_J$ and $V_{J-1}$, respectively, given by
\begin{equation}
\label{etaj}
\begin{aligned}
 \Aa[J][\eta_J][v]&= \int_{\Omega}\phi v\ dx
&&\forall v\in V_J,\\
 \Aa[J-1][\eta_{J-1}][v]&= \int_{\Omega}\phi v\ dx 
&&\forall v\in V_{J-1}.
\end{aligned}
\end{equation}
By Theorem~\ref{thm:errors} and the hypotheses \eqref{hjhj1} and \eqref{pjpj1}, 
we deduce that
\begin{equation}
\begin{aligned}
\normL[\eta-\eta_J]&\lesssim \CLtwonew{J}\frac{\h{J-1}^2}{p_{J-1}^{\textcolor{black}{2-\mu}}}\normH[\eta][2][\Om],\\
\normL[\eta-\eta_{J-1}]&\lesssim \CLtwonew{J-1}\frac{\h{J-1}^2}{p_{J-1}^{\textcolor{black}{2-\mu}}}\normH[\eta][2][\Om],\label{dualerrL22}
\end{aligned}
\end{equation}
and from a standard elliptic regularity assumption, it follows that
\begin{equation}
\begin{aligned}
\normL[\eta-\eta_J]&\lesssim \CLtwonew{J}\frac{\h{J-1}^2}{p_{J-1}^{\textcolor{black}{2-\mu}}}\normL[\phi],\\
\normL[\eta-\eta_{J-1}]&\lesssim \CLtwonew{J-1}\frac{\h{J-1}^2}{p_{J-1}^{\textcolor{black}{2-\mu}}}\normL[\phi].\label{dualerrL2}
\end{aligned}
\end{equation}
Recalling the definition of $P_{J}^{J-1}$, cf. \eqref{project},  and \eqref{etaj}, for any $w\in V_{J-1}$, we get
$$
\Aa[J-1][P_J^{J-1}\eta_J][w] 
= \Aa[J][\eta_J][I^J_{J-1}w] = \Aa[J][\eta_J][w] = \int_\Omega\phi w\ dx = \Aa[J-1][\eta_{J-1}][w].
$$
Hence, 
\begin{equation}
\eta_{J-1} = P_J^{J-1}\eta_J.
\label{ejPej1}
\end{equation}
According to \cite[Lemma~4.1]{AntSarVer15}, the following generalized Cauchy-Schwarz inequality holds
\begin{equation}
\Aa[J][v][w]\leq \ltrivert v \rtrivert_{0,J}\ltrivert w\rtrivert_{2,J},
\label{Csgen}
\end{equation}
for any $v,w\in V_J$. We now employ \eqref{etaj} and the definition of $P^{J-1}_J$ in \eqref{project}, followed by \eqref{ejPej1}, the Cauchy-Schwarz inequality \eqref{Csgen} and the error estimates \eqref{dualerrL2}, to get
\begin{align}
\int_\Omega \phi(\textnormal{Id}_J-I_{J-1}^JP^{J-1}_J)v\ dx 
=& ~ \Aa[J][\eta_J][v]-\Aa[J][\eta_J][I_{J-1}^JP^{J-1}_Jv]\\
=& ~ \Aa[J][\eta_J][v]-\Aa[J-1][P^{J-1}_J\eta_J][P^{J-1}_Jv]\\
=& ~ \Aa[J][\eta_J][v]-\Aa[J-1][\eta_{J-1}][P^{J-1}_Jv]\\
=& ~ \Aa[J][\eta_J-I_{J-1}^J\eta_{J-1}][v]\\
\leq& ~\ltrivert \eta_J-\eta_{J-1} \rtrivert_{0,J}\ltrivert v \rtrivert_{2,J}\\
\leq & ~(\normL[\eta_J-\eta]+\normL[\eta_{J-1}-\eta])\ltrivert v \rtrivert_{2,J}\\
\lesssim &~(\CLtwonew{J} +\CLtwonew{J-1})\frac{\h{J-1}^2}{p_{J-1}^{\textcolor{black}{2-\mu}}}\normL[\phi]\ltrivert v \rtrivert_{2,J}.
\label{bdd:lem_approx_twolevel}
\end{align}
Substituting \eqref{bdd:lem_approx_twolevel} into \eqref{dualratio} gives the desired result.
\end{proof}

The convergence result for the two-level method, involving the error propagation operator $\mathbb{E}^{\mathsf{2lvl}}_{m_1,m_2}$ 
defined in \eqref{E2lvl}, is obtained by combining Lemma~\ref{lem:smooth} and Lemma~\ref{lem:approx}.

%%%%%
\begin{theorem}
\label{thm: 2lvl}
\textcolor{black}{Assume that Assumptions~\ref{hyp:grid_new},~\ref{hyp:shapereg},~\ref{hyp:shape-regular-element},~and \ref{hyp:covering} hold.} \textcolor{black}{Let $\mu$ be defined as in Lemma \ref{lem:interpDG}.} Then,
there exists a  positive constant $\Ctwolvl$ independent of the mesh size and the polynomial approximation degree, such that
\begin{equation}
\label{2lvl}
\ltrivert\mathbb{E}^{\mathsf{2lvl}}_{m_1,m_2}v\rtrivert_{1,J} \leq \Ctwolvl\Sigma_J \ltrivert v\rtrivert_{1,J},
\end{equation}
for any $v\in V_J$, with 
\begin{equation}
\label{sigmaJ}
\Sigma_J=\Ctildenew{J}{J-1}\frac{p_{J}^{2\textcolor{black}{+\mu}}}{(1+m_1)^{1/2}(1+m_2)^{1/2}},
\end{equation}
where $\Ctildenew{J}{J-1} = \Ceig{J}(\CLtwonew{J}+\CLtwonew{J-1})$. Therefore, the two-level method converges uniformly provided the number of pre- and post-smoothing steps satisfy
\begin{equation}
\label{2lvlmbound}
(1+m_1)^{1/2}(1+m_2)^{1/2}\geq \chi \Ctildenew{J}{J-1}p_{J}^{2\textcolor{black}{+\mu}},
\end{equation}
for a positive constant $\chi>\Ctwolvl$.
\end{theorem}
\begin{proof}
The statement of the theorem follows in a straightforward manner by applying the smoothing property \eqref{smoothingprop} 
twice, the approximation property \eqref{approxprop} and exploiting the bounded variation assumptions \eqref{hjhj1} and \eqref{pjpj1}.
\end{proof}
\added{We observe that the geometric properties of the partitions are hidden in the constant $\Ctildenew{J}{J-1}$. As a consequence, a good quality agglomerated coarse grid is fundamental to guarantee a mild condition on the minimun number of smoothing steps.} 
\par

\subsection{Convergence of the W-cycle multigrid algorithm} \label{sec:wcycle_analysis}

We first need to establish the equivalence between DG norms on subsequent grid levels. We point out that,
in contrast to the case of standard quasi-uniform grids presented in \cite{AntSarVer15}, such an equivalence result
does not follow in a straightforward manner; indeed, here we need to exploit Assumption~\ref{ass:agglomeration_new} 
introduced in the previous section. Under these assumptions, the proof of the following result follows immediately.
%%%%%%%%%%%%%%%%%%
\begin{lemma}
\label{lem:DGequiv}
Assuming Assumption~\ref{ass:agglomeration_new} holds, then for any $v\in V_{j-1}$, $j = 2,\dots,J$, we have that
\begin{equation}
\label{DGequiv}
\normDG[v][j]\leq \Cequiv \normDG[v][j-1],
\end{equation}
where $\Cequiv = \Cequiv(\Theta)$, in general, depends on the quality of the agglomerated grids.
\end{lemma}\par

Lemma~\ref{lem:DGequiv} is essential to deduce the  stability of the operators $I_{j-1}^j$ and $P_{j}^{j-1}$,
$j = 2,\dots,J$. In particular, we state the following bounds.
\begin{lemma}
\label{lem: RPstab}
\textcolor{black}{Assuming Assumption~\ref{ass:agglomeration_new} holds}, then there exists a constant \textcolor{black}{$\Cstab \geq 1$}, independent of the mesh size, the polynomial 
approximation degree and the level $j$, $j = 2,\dots,J$, such that
\begin{align}
\ltrivert I_{j-1}^j v\rtrivert_{1,j} &\leq \Cstab\ltrivert  v\rtrivert_{1,j-1}
&& \forall v\in V_{j-1},\label{Rstab}\\
\ltrivert P^{j-1}_j v\rtrivert_{1,j-1} &\leq \Cstab\ltrivert  v\rtrivert_{1,j}&&
\forall v\in V_{j}.\label{Pstab}
\end{align}
\end{lemma}
The proof of Lemma~\ref{lem: RPstab} is based on employing inequality \eqref{DGequiv}; for details, see \cite[Lemma~4.6]{AntSarVer15}.
\begin{remark}
We stress that the constant $\Cstab$ depends on $\Cequiv$ in \eqref{DGequiv}, which means that the quality of the agglomerated meshes plays a crucial role in keeping this constant bounded, thus resulting in the uniformity with respect to the mesh size and the number of levels as shown in Theorem~\ref{thm:final} below.
\end{remark}\par
The error propagation operator associated to Algorithm~\ref{alg:multilevel} is defined as 
\begin{equation}
\left\{
\begin{aligned} 
\mathbb{E}_{1,m_1,m_2}v &= 0\\
\mathbb{E}_{j,m_1,m_2}v &= G_j^{m_2}(\textnormal{Id}_j-I_{j-1}^j(\textnormal{Id}_j-\mathbb{E}_{j-1,m_1,m_2}^2)P_j^{j-1})G_j^{m_1}v,
\ j = 2,\dots,J,
\end{aligned}
\label{defEj}
\right.
\end{equation}
where $G_j = \textnormal{Id}_j-B_j^{-1}A_j$ and $P_j^{j-1}$ is defined analogously to \eqref{project}, cf. \cite{Hack,bramble}.
Then the convergence estimate for the W-cycle  multigrid scheme follows from Theorem~\ref{thm: 2lvl} and the stability estimates \eqref{Rstab} and \eqref{Pstab}.
%%%%
\begin{theorem}
\label{thm:final}
\textcolor{black}{Assume that Assumptions~\ref{hyp:grid_new},~\ref{hyp:shapereg},~\ref{hyp:shape-regular-element},~\ref{hyp:covering}, and~\ref{ass:agglomeration_new} hold.} \textcolor{black}{Let $\mu$ be defined as in Lemma \ref{lem:interpDG}.}
Let $\Ctwolvl$ and $\Cstab$ be defined as in Theorem~\ref{thm: 2lvl} and Lemma~\ref{lem: RPstab},
respectively, \added{and let $\Ctildenew{j}{j-1}$ be defined as in Theorem~\ref{thm: 2lvl}, but on the level $j$, i.e., $\Ctildenew{j}{j-1}=\Ceig{j}(\CLtwonew{j}+\CLtwonew{j-1})$, $j=2,\ldots, J$.} Then, there exists a constant $\Chat>\Ctwolvl$, independent of the mesh size, the polynomial approximation degree and the level $j$, $j = 1,\dots,J$, such that, if the number of pre- and post-smoothing steps satisfy
\begin{equation}
\label{mhat}
(m_1+1)^{1/2}(m_2+1)^{1/2}\geq\left\{
\begin{aligned}
&p_{j}^{2\textcolor{black}{+\mu}}\Ctildenew{j}{j-1}\frac{(\Cstab)^2\Chat^2}{\Chat-\Ctwolvl}\quad && \textnormal{if }
\Ctildenew{j-1}{j-2}\leq \Ctildenew{j}{j-1},\\
&p_{j}^{2\textcolor{black}{+\mu}}\frac{(\Ctildenew{j-1}{j-2})^2}{\Ctildenew{j}{j-1}}\frac{(\Cstab)^2\Chat^2}{\Chat-\Ctwolvl}\quad &&\textnormal{otherwise,}
\end{aligned}\right.
\end{equation}
then
\begin{equation}
\label{final}
\ltrivert \mathbb{E}_{j,m_1,m_2}v\rtrivert_{1,j}\leq\Chat\Sigma_j\ltrivert v\rtrivert_{1,j}\quad \forall v\in V_j,
\end{equation}
with
\begin{equation}
\label{sigmaj}
\Sigma_j=\Ctildenew{j}{j-1}\frac{p_{j}^{2\textcolor{black}{+\mu}}}{(1+m_1)^{1/2}(1+m_2)^{1/2}}.
\end{equation} 
\end{theorem}
%%%%%%%%%%%
\begin{proof}
The proof follows the derivation of the analogous result presented in \cite[Theorem~4.7]{AntSarVer15}. 
For $j = 1$, the statement of the theorem trivially holds. For $j>1$, by an induction hypothesis, 
we assume that \eqref{final} holds for $j-1$. By the definition of the error propagation operator 
$\mathbb{E}_{j,m_1,m_2}v$ in \eqref{defEj}, it follows that
\begin{align*}
\ltrivert \mathbb{E}_{j,m_1,m_2}v\rtrivert_{1,j}
\leq& ~\ltrivert G_j^{m_2}(\textnormal{Id}_j-I_{j-1}^jP_j^{j-1})G_j^{m_1}v\rtrivert_{1,j}\\
&+\ltrivert G_j^{m_2}I_{j-1}^j\mathbb{E}_{j-1,m_1,m_2}^2P_j^{j-1}G_j^{m_1}v\rtrivert_{1,j}.
\end{align*}
The first term corresponds to a two-level method between level $j$ and $j-1$. We now observe that the smoothing property of Lemma~\ref{lem:smooth} and the approximation property of Lemma~\ref{lem:approx} can be  extended to any level $V_j$, $j=2,\dots,J$, and we therefore have, by Theorem~\ref{thm: 2lvl}, that
$$\ltrivert G_j^{m_2}(\textnormal{Id}_j-I_{j-1}^jP_j^{j-1})G_j^{m_1}v\rtrivert_{1,j} \leq \Ctwolvl\Sigma_j \ltrivert v\rtrivert_{1,j}.$$
The bound on the second term is obtained by applying the smoothing property \eqref{smoothingprop} for $j=2,\ldots,J$, the stability estimates \eqref{Rstab} and \eqref{Pstab} and the induction hypothesis; thereby, we get
\begin{align*}
\ltrivert G_j^{m_2}I_{j-1}^j\mathbb{E}_{j-1,m_1,m_2}^2P_j^{j-1}G_j^{m_1}v\rtrivert_{1,j}\leq
&(\Cstab)^2\Chat^2\Sigma_{j-1}^2\ltrivert v\rtrivert_{1,j}.
\end{align*}
We then obtain
\begin{equation*}
\ltrivert \mathbb{E}_{j,m_1,m_2}v\rtrivert_{1,j}\leq\left(\Ctwolvl\Sigma_j+(\Cstab)^2\Chat^2\Sigma_{j-1}^2\right)\ltrivert v\rtrivert_{1,j}.
\end{equation*}
\textcolor{black}{
We now observe that the following relation between $\Sigma_{j-1}$ and $\Sigma_{j}$ holds 
\begin{align*}
\Sigma_{j-1}=\Sigma_{j}\left( \frac{p_{j-1}}{p_j} \right) \left( \frac{\Ctildenew{j-1}{j-2}}{\Ctildenew{j}{j-1}} \right) 
\leq  \Sigma_{j}\left( \frac{\Ctildenew{j-1}{j-2}}{\Ctildenew{j}{j-1}} \right).
\end{align*}
Using the above identity we have that
\begin{align*}
&\Ctwolvl\Sigma_j+(\Cstab)^2\Chat^2\Sigma_{j-1}^2  \\
&\qquad \qquad \leq \left(\Ctwolvl+(\Cstab)^2\Chat^2  \frac{(\Ctildenew{j-1}{j-2})^2}{\Ctildenew{j}{j-1}} 
\frac{p_{j}^{2\textcolor{black}{+\mu}}}{(1+m_1)^{1/2}(1+m_2)^{1/2}}\right)\Sigma_j.
\end{align*}
We then observe that if $m_1$ and $m_2$ are such that
\begin{equation*}
(1+m_1)^{1/2}(1+m_2)^{1/2} \geq p_{j}^{2\textcolor{black}{+\mu}}\frac{(\Ctildenew{j-1}{j-2})^2}{\Ctildenew{j}{j-1}}\frac{(\Cstab)^2\Chat^2}{\Chat-\Ctwolvl},
\end{equation*}
it follows that $\Ctwolvl\Sigma_j+(\Cstab)^2\Chat^2\Sigma_{j-1}^2\leq \Chat\Sigma_j$.
Notice that for $\Ctildenew{j-1}{j-2}\leq \Ctildenew{j}{j-1}$ the above condition on $m_1$ and $m_2$ can be simplified as follows
\begin{equation*}
\label{mbound}
(1+m_1)^{1/2}(1+m_2)^{1/2}\geq p_{j}^{2\textcolor{black}{+\mu}}\Ctildenew{j}{j-1}\frac{(\Cstab)^2\Chat^2}{\Chat-\Ctwolvl},
\end{equation*}
and therefore we obtain $\Ctwolvl\Sigma_j+(\Cstab)^2\Chat^2\Sigma_{j-1}^2\leq \Chat\Sigma_j$, and the proof is complete.
}
\end{proof}

\textcolor{black}{
As in the two-level case, inequality \eqref{final} implies that  the convergence of the method is guaranteed if the 
number of smoothing steps is chosen sufficiently large, cf. \eqref{mhat}. Moreover, compared to the case of standard 
quasi-uniform grids, cf. \cite{AntSarVer15}, the bound \eqref{mhat} on the number of smoothing steps involves a 
dependence on the geometric properties of the underlying agglomerated meshes, which in principle, could lead to  
shape-regularity conditions on the hierarchy of grids employed. 
However, we remark that, in practice, the numerical simulations indicate that the proposed multigrid algorithms 
converge uniformly, even when low quality agglomerated grids are employed; moreover, an increase in the polynomial 
order does not seem to require a higher number of smoothing steps to obtain a convergent iteration, cf. 
Section~\ref{sec:numerical} for details.
}
\begin{remark}
\label{rem:theta}
Whenever the agglomerated grids are not quasi-uniform, Theorem~\ref{thm: 2lvl} and Theorem~\ref{thm:final} still hold. More precisely, we need to introduce the ratio $\theta_j$ between the maximum and minimum element size on level $j$, i.e., 
\begin{equation}
\theta_j = \frac{\max_{\elem\in\mcal[T][j]}\h{\elem}}{\min_{\elem\in\mcal[T][j]}\h{\elem}},\quad j=1,\dots,J.
\end{equation}
Assuming there exists a constant $\Cmesh$, independent of the granularity of the mesh, such that
\begin{equation}
\label{Cmesh}
\theta_j\leq \Cmesh, \quad j=1,\dots,J,
\end{equation}
then the results in Theorem~\ref{thm: 2lvl} and Theorem~\ref{thm:final} hold with
\begin{equation*}
\Sigma_j=\theta_j^2\Ctildenew{j}{j-1}\frac{p_{j}^{2\textcolor{black}{+\mu}}}{(1+m_1)^{1/2}(1+m_2)^{1/2}},
\end{equation*}
cf. \eqref{sigmaj}. \paul{Moreover,} the bound \eqref{mhat} is modified as follows
\begin{equation}
\label{mhattheta}
(1+m_1)^{1/2}(1+m_2)^{1/2}\geq\left\{
\begin{aligned}
&\frac{(\Cstab)^2\Chat^2}{\Chat-\Ctwolvl}\frac{(\Cmesh)^4}{\theta_j^2}\Ctildenew{j}{j-1}p_{j}^{2\textcolor{black}{+\mu}}\quad && \textnormal{if }\Ctildenew{j-1}{j-2}\leq \Ctildenew{j}{j-1},\\
&\frac{(\Cstab)^2\Chat^2}{\Chat-\Ctwolvl} \frac{(\Cmesh)^4}{\theta_j^2}\frac{(\Ctildenew{j-1}{j-2})^{2}}{\Ctildenew{j}{j-1}}p_{j}^{2\textcolor{black}{+\mu}}\quad &&\textnormal{otherwise.}
\end{aligned}\right.
\end{equation}
\end{remark}

\begin{remark}
\textcolor{black}{
We recall that in Theorem~\ref{thm:final} and Remark~\ref{rem:theta}, in order to guarantee the convergence of the method, we require a lower bound on the number of smoothing steps, cf. \eqref{mhat} and \eqref{mhattheta}.
Such a condition guarantees that the resulting W-cycle method is uniformly convergent with 
respect to the mesh size, the polynomial approximation degree, and the number of levels. 
In fact, for $\Ctildenew{j-1}{j-2}\leq \Ctildenew{j}{j-1}$, from \eqref{mhattheta} and 
using that $\theta_j\leq \Cmesh$, $j=1,\dots,J$, we obtain
\begin{equation*}
\Chat\Sigma_j=\Chat\theta_j^2\Ctildenew{j}{j-1}\frac{p_{j}^{2\textcolor{black}{+\mu}}}{(1+m_1)^{1/2}(1+m_2)^{1/2}}
\leq \frac{\Chat-\Ctwolvl}{(\Cstab)^2\Chat}\frac{\theta_j^4}{(\Cmesh)^4}
\leq \frac{\Chat-\Ctwolvl}{(\Cstab)^2\Chat}< 1.
\end{equation*}
An analogous result can be obtained for $\Ctildenew{j-1}{j-2}> \Ctildenew{j}{j-1}$. Moreover, we note that we have considered the general setting of \eqref{mhattheta}, since \eqref{mhat} can be regarded as a particular case. 
}
\end{remark}

%%%%%%%%%%%%%%%%
\textcolor{black}{
\section{Weaker geometric assumptions on the quality of the agglomerates}
\label{rem:old_assumption_rev}}
\textcolor{black}{
%\todo{Note that shape regularity of the mesh is still required to ensure that Lemma 4 holds for elements which are not coverable.}
\paul{In this section we briefly provide some details regarding the minimal regularity 
requirements needed to guarantee that our geometric $h-$multigrid method is convergent. 
Indeed, the theoretical analysis of our two-level and W-cycle multigrid algorithms solver can 
be undertaken under weaker mesh assumptions on the shape of the elements and the quality 
of the agglomerated grids than those satisfying Assumptions~\ref{hyp:grid_new} 
and~\ref{hyp:shape-regular-element}. Before we proceed, let us first introduce the
following two definitions.
\begin{definition}\label{def:covering}
An element $\kappa\in\mcal[T][j]$, $j=1,\dots,J$, is said {\em $p_j$-coverable} with 
respect to $p_j\in\mathbb{N}$, if there exists a set of $l_{\kappa}$ shape-regular 
simplices $K_i$, $i=1,\dots, l_{\kappa}$, $l_{\kappa}\in\mathbb{N}$,
such that
\begin{equation}
\mbox{dist}(\kappa, \partial K_i) <  C_{as}\frac{\mbox{diam}(K_i)}{p_j^{2}},
\qquad
\mbox{and} 
\qquad
|K_i|\ge c_{as} |\kappa| \nonumber
\end{equation}
for all $i=1,\dots, l_{\kappa}$, where $C_{as}$ and $c_{as}$ are
positive constants, independent of $\kappa$ and $\mcal[T][j]$.
\end{definition}
%%%%%%%%%%%%%%%
\begin{definition}
\label{def:kb}
For each $\elem\in\mcal[T][j]$, $j=1,\dots,J$, we denote by \mcal[F][\elem][\flat] the set of all possible $d$-simplices contained in \elem and having at least one face in common with \elem. Moreover, we denote by $\elem_F^\flat$, an element in \mcal[F][\elem][\flat] sharing a face $F$ with $\elem\in\mcal[T][j]$.
\end{definition}
%%%%%%%%%%%%%
We point out that, assuming each mesh $\mcal[T][j]$, $j=1, \ldots, J$, is shape-regular,
then Lemma~\ref{lem:inverse} can be shown to hold, without the need to assume
that Assumption~\ref{hyp:shape-regular-element} is satisfied for elements
which are $p_j$-coverable; see \cite{CangianiDongGeorgoulisHouston_2016} for details.
Secondly, as an alternative to Assumption~\ref{hyp:grid_new}, we 
may consider the following condition.}
%%%%%%%%%%%%
\begin{assumption}{(Weaker mesh regularity assumptions)}\label{ass:weak}
For any $j=1, \ldots, J$, the mesh $\mcal[T][j]$ satisfies the following regularity properties
\begin{enumerate}[label={\bf \ref{ass:weak}.\alph*}, leftmargin=2\parindent]
\item \label{ass:weak_CF} The number of faces of any $\elem\in\mcal[T][j]$, $j=1,\dots,J$, is uniformly bounded;
\item \label{ass:agglomeration} For any $F\in\mcal[F][j]\cap\mcal[F][j-1]$, $j=2,\ldots,J$, we denote by $\elem^\pm_{j}$ and $\elem^\pm_{j-1}$ the neighboring elements sharing the face $F$ in \mcal[T][j] and \mcal[T][j-1], respectively. We assume that there exists $\Theta>0$ such that
\begin{equation*}
\begin{aligned}
& 1<\frac{|\elem_{j-1}^\pm|}{|\elem_j^\pm|}\leq \Theta\quad\forall F\in\mcal[F][j]\cap\mcal[F][j-1]
&& \textrm{and}
&& \frac{|\elem_j^\pm|}{\sup_{\elem_F^\flat\in\elem_j^\pm}|\elem_F^\flat|}\approx \frac{|\elem_{j-1}^\pm|}{\sup_{\elem_F^\flat\in\elem_{j-1}^\pm}|\elem_F^\flat|}.
\end{aligned}
\end{equation*} 
\end{enumerate}
\end{assumption}
%%%%%%%%%%%%
Assumption~\ref{ass:weak_CF} might in principle be critical in our multilevel framework, because the number of faces grows with the number of levels due to the agglomeration process. As a consequence, this assumption is only satisfied if the number of levels is 
kept limited. However, it will be demonstrated in Section~\ref{sec:numerical}, that this assumption \paul{only} seems to be required \paul{from a theoretical point of view.} \\
}
\textcolor{black}{
A key step in the weakening of the mesh conditions is establishing
an inverse inequality of the form outlined in 
Lemma~\ref{lem:inversepoly}, which holds for general polygonal/polyhedral elements.  
Indeed, assuming Assumption~\ref{ass:weak_CF} is satisifed, 
then following inverse inequality 
holds, cf. \cite[Lemma~4.4]{Canetal14}.
%%%%%%%%%%%%%%
\begin{lemma} \label{lem:weak_inverse}
Let $\elem\in\mcal[T][j]$, $j=1,\dots,J$, be a polygonal/polyhedral element, and let $F\subset\partial\elem$ be one of its faces. Then, for each $v\in\mcal[P][p_j](\elem)$, 
we have
\begin{equation}
\label{inversepoly_old}
\normL[v][2][F][2]\leq 
\CINV{j}{p_j}{\elem}{F}
\frac{p_j^2|F|}{|\elem|}\normL[v][2][\elem][2],
\end{equation}
with
\begin{equation}
\CINV{j}{p_j}{\elem}{F}=C\left\{\begin{aligned}
\displaystyle \min&\left\{\frac{|\elem|}{\sup_{\elem_F^\flat\subset\elem}|\elem_F^\flat|},p_j^{2d}\right\},\quad &&\textnormal{if }\elem \textnormal{ is $p_j$-coverable},\\
\displaystyle&\phantom{\left\{\right\}}\frac{|\elem|}{\sup_{\elem_F^\flat\subset\elem}|\elem_F^\flat|}, &&\textnormal{ otherwise},
\end{aligned}\right.
\end{equation}
and $\elem_F^\flat\in \mcal[F][\elem][\flat]$ as in Definition~\ref{def:kb}. The positive constant $C$ is independent of $|\elem|/\sup_{\elem_F^\flat\in\elem}|\elem_F^\flat|$, $p_j$ and $v$.
\end{lemma}
\paul{Equipped with Lemma~\ref{lem:weak_inverse}, the interior penalty stabilization 
function $\sigma_j$, must be appropriately redefined; see \cite{Canetal14} 
for details.}
Finally, we observe that Assumption~\ref{ass:agglomeration}, 
together with \eqref{pjpj1}, implies that 
\begin{equation}
\begin{aligned}
\label{Cinvequiv}
\CINV{j}{p_j}{\elem_{j}^\pm}{F}\approx\CINV{j-1}{p_{j-1}}{\elem_{j-1}^\pm}{F}
&& \forall\, F\in\mcal[F][j]\cap\mcal[F][j-1],
&& j=2,\ldots,J.
\end{aligned}
\end{equation}
}
%%%%%%%%%%  
\section{Numerical results}
\label{sec:numerical}
In this section we present several numerical simulations to verify the theoretical estimates provided in 
Theorem~\ref{thm: 2lvl} and Theorem~\ref{thm:final} in the case of a two dimensional problem on the unit 
square $\Omega = (0,1)^2$. 
For the numerical tests, we consider the \added{two} sets of meshes shown in 
Figures~\ref{fig:grids} and~\ref{fig:tria_grids}.
 \added{The first set of initial grids are shown in Figure~\ref{fig:grids} (top line) and consist of 512 (Set 1), 1024 (Set 2), 2048 (Set 3) and 4096 (Set 4) polygonal elements. These meshes have been generated using the software package \texttt{PolyMesher} \cite{Talietal12}. We also consider an initial set of decompositions constisting of 582 (Set 1), 1086 (Set 2), 2198 (Set 3) and 4318 (Set 4) shape-regular triangles as \paul{depicted} in Figure~\ref{fig:tria_grids} (top line).}
Each initial grid is then subsequently coarsened in order to obtain a sequence of nested partitions by employing the software package 
MGridGen \cite{MouKar01,MouKar201}. \\
%%%%%%%%%%%%%%%%%%%%%%%%%%%%%%%
\begin{figure}[t!]
\centering
\includegraphics[width=1.0\textwidth]{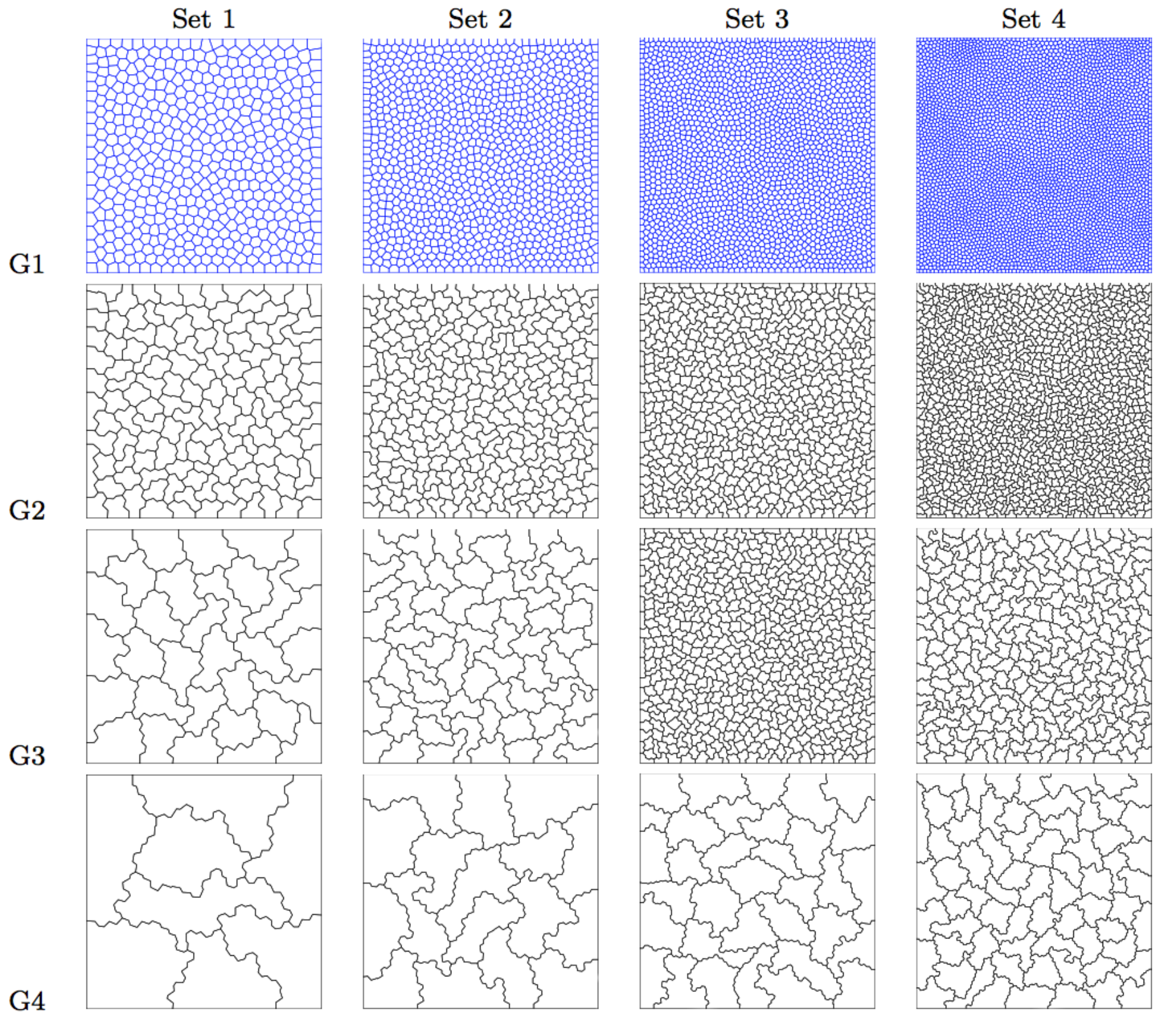} 
\caption{Sequences of agglomerated grids employed for numerical simulations. 
Top line: fine grids consisting of 512 (Set 1), 1024 (Set 2), 2048 (Set 3) and 4096 (Set 4) polygonal elements.}
\label{fig:grids}
\end{figure}

\begin{figure}[t!]
\centering
\includegraphics[width=1.0\textwidth]{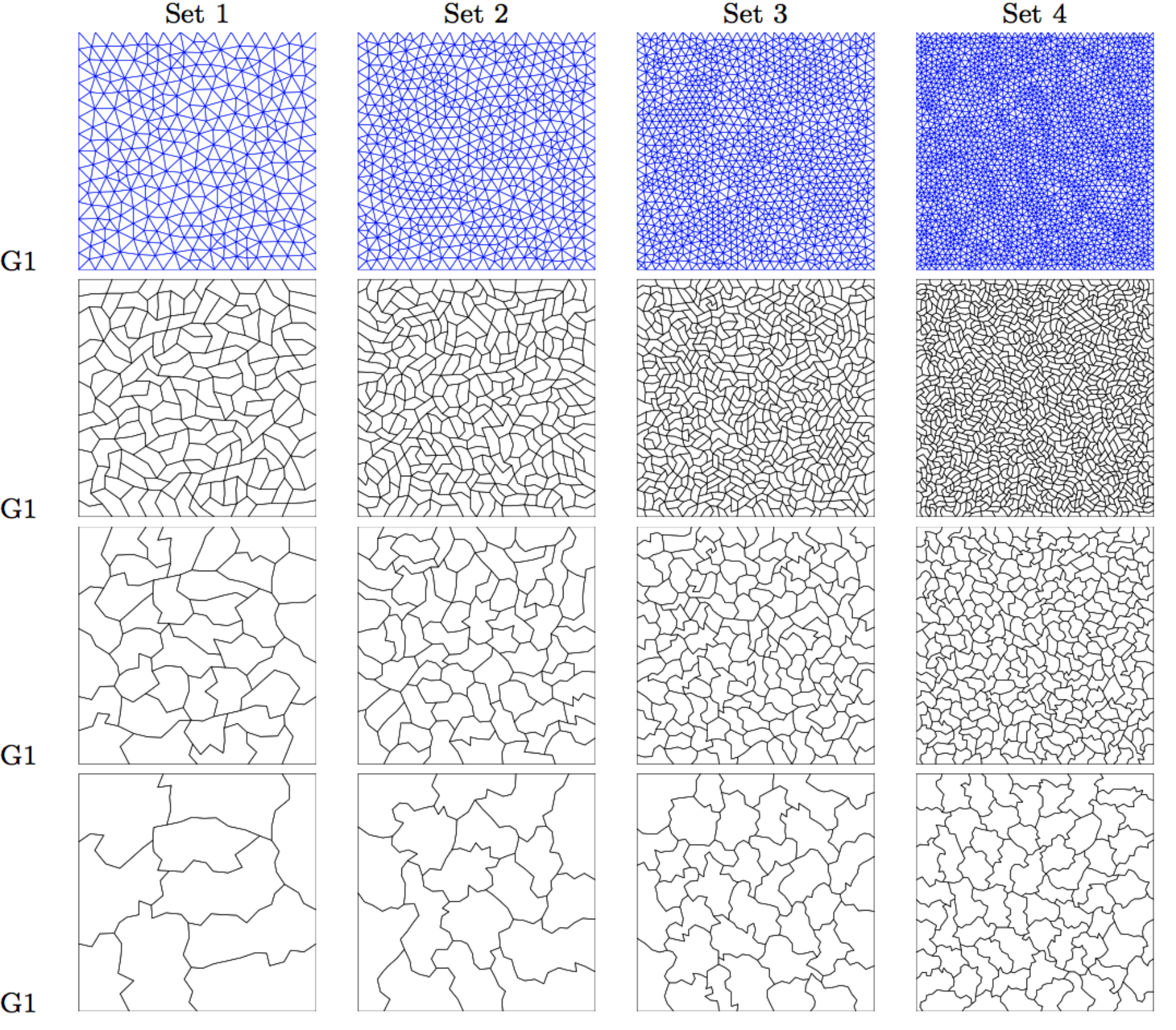} 
\caption{Sequences of agglomerated grids employed for numerical simulations. Top line: fine grids consisting of 582 (Set 1), 1086 (Set 2), 2198 (Set 3) and 4318 (Set 4) triangular elements.}
\label{fig:tria_grids}
\end{figure}

%%%%%%%%%%%%%%%%%%%%%%%%%%%%%%%

\added{Before testing the performance of the two-level and W-cycle multigrid solvers presented in Algorithm~\ref{alg:2lvl} and Algorithm~\ref{alg:multilevel}, respectively,  we first address the 
issue of the choice of the penalization coefficient $C_{\sigma}^j$ in \eqref{bilinearj}. 
According to Lemma~\ref{lem:contcoerc}, the bilinear form \Aa is coercive provided that $C_\sigma^j$ is chosen 
large enough. In Table~\ref{tab:Ccoerc}, we report the coercivity constant $\Ccoer$ of \eqref{coerc} for a fixed 
value of $C_\sigma^j \equiv C_\sigma = 10$ for $j=1,\ldots,4$.   We observe that the bilinear form is uniformly 
coercive for a constant value of the penalization coefficient, which is of the same magnitude as the one typically 
employed on standard shape-regular triangular meshes. As a consequence, in the following, we set $C_\sigma^j \equiv C_\sigma = 10$ for $j=1,\dots,4$.}
%%%%%%%%%%%%%%%%%%%%%%%%%%%%%%%%
\begin{table}[t]
\centering
\footnotesize
\begin{tabular}{lcccc}
\hline
& Set 1 & Set 2 & Set 3 & Set 4\\
\hline
G1 & 0.7385 & 0.7375 & 0.7370 & 0.7364\\
G2 & 0.7624 & 0.7564 & 0.7559 & 0.7545\\
G3 & 0.7827 & 0.7818 & 0.7720 & 0.7611\\
G4 & 0.8153 & 0.8054 & 0.8001 & 0.7827\\
\hline
\end{tabular}
\caption{Value of the coercivity constant $\Ccoer$ for the sets of grids considered in Figure~\ref{fig:grids} with $C_{\sigma}^j = C_{\sigma} = 10$, $j = 1,\dots$,4. }
\label{tab:Ccoerc}
\end{table}
%%%%%%%%%%%%%%%%%%%%%%%%%%%%%%%%%%%%%

\added{We now consider the sequence of the grids shown in Figure~\ref{fig:grids},  Set 1}, and numerically evaluate the constant $\Ctwolvl\Sigma_J$, $J=2$, in 
Theorem~\ref{thm: 2lvl} and  the constant $\Chat\Sigma_3$ in Theorem~\ref{thm:final}, for the $h$-version 
of the two solvers, based on selecting $m_1=m_2=m=2p^2$, cf. Figure~\ref{fig:mp2}. Here, we observe that 
$\Ctwolvl\Sigma_2$ and $\Chat\Sigma_3$ are roughly (asymptotically) constant, as the polynomial
degree $p$ increases; thereby, this implies that $\Ctildenew{J}{J-1}$, $J=2,3$, respectively, is approximately $\mathcal{O}(1)$,
as $p$ increases. \textcolor{black}{Notice also that, in practice, the parameter $\mu=0$, even whenever a $p$--optimal interpolant cannot be explicetely constructed.}
%%%%%%%%%%%%%%%%%%%%%%%%%%%%%%%%%%
\begin{figure}[t!]
\centering
\includegraphics[scale=1.5]{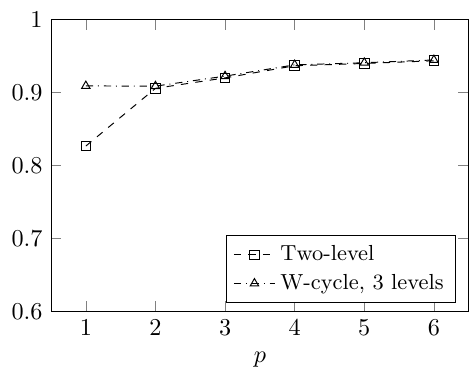}
\caption{Estimates of $\Ctwolvl\Sigma_J$ and $\Chat\Sigma_3$ in \protect\eqref{2lvl} 
and \protect\eqref{final}, respectively, as a function of $p$, and $m_1=m_2=m=2p^2$. \added{Sequence of agglomerated meshes shown in Figure~\ref{fig:grids}, Set 1.}}
\label{fig:mp2}
\end{figure}
%%%%%%%%%%%%%%%%%%%%

Next, we investigate the performance of the two-level and W-cycle multigrid schemes in terms of the convergence factor
\begin{equation*}
\rho = \exp\left(\frac{1}{N}\ln \frac{\|\mathbf{r}_{N}\|_2}{\|\mathbf{r}_{0}\|_2}\right),
\end{equation*}
where $N$ denotes the number of iterations required to attain convergence up to a (relative) tolerance of $10^{-8}$ 
and $\mathbf{r}_{N}$ and $\mathbf{r}_{0}$ are the final and initial residual vectors, respectively. In Table~\ref{tab:hMGvsm}, 
we report the iteration counts and the convergence factor (in parenthesis), needed to attain convergence of the $h$-version 
of the two-level (TL) method and W-cycle multigrid scheme (with 3 and 4 levels), as a function of the number of elements 
(given by the choice of different grid sets), and the number of smoothing steps ($m_1=m_2=m$). Here, we have fixed the polynomial approximation order on each level $p_j \equiv p=1$. We first observe that, although the agglomerated grids, in general, do not 
necessarily strictly satisfy \paul{Assumption~\ref{hyp:shapereg}}, the number of iterations, for fixed $m$, does not 
significantly increase with the number of elements in the underlying mesh; moreover, for the W-cycle solver, the number of iterations remains bounded with the number of levels. As expected, the convergence is faster for larger values of $m$ and the solvers are convergent provided the number of smoothing steps is sufficiently large. For each grid, we have also reported the iteration counts $\NCG$ for the Conjugate Gradient (CG) method, which shows that the two proposed solvers outperform the CG scheme in terms of the number of iterations required to attain convergence, even when a small number of smoothing steps are employed.
\textcolor{black}{For the sake of comparison, we also report the iteration counts $\NPCG$ for the Preconditioned Conjugate Gradient (PCG) method, based on employing a simple block Jacobi preconditioner. The extension to polytopic grids of the domain decomposition preconditioning techniques, such as, for example, the ones proposed in \cite{AntHou,AntHouSme,AntSarVerZik_2016}, in the DG setting, or in \cite{Pav94,SchoMelPecZag_2008}, in the conforming setting, are currently under investigation and will be the subject of future research.} 
Table~\ref{tab:hMGvsm3} presents analogous results for the first three sets of meshes, in the case when $p=3$. Here, we observe that, as expected, the convergence factor increases, but the increase in $p$ does not require an increase in the minimal number of smoothing steps 
needed to ensure that the underlying multilevel solvers are convergent.
%%%%%%%%%%%%%%%%%%%%%%%%%%%%%%%%%%%%%
\begin{table}[t!]
\setlength\extrarowheight{1.5pt}
\centering
\footnotesize
\begin{tabular}{lccc||ccc}
\hline
 & \multicolumn{3}{c||}{Set 1} & \multicolumn{3}{c}{Set 2}\\ 
\hline
 & \multirow{2}{*}{TL}&\multicolumn{2}{c||}{W-cycle} & \multirow{2}{*}{TL}&\multicolumn{2}{c}{W-cycle}\\ 
 \hhline{~~--~--}
 & & 3 lvl & 4 lvl & &  3 lvl & 4 lvl\\
\hline
$m=3$ & 133 (0.87)& 160 (0.89)& 167 (0.90) & 121 (0.86)& 191 (0.91)& 188 (0.91)\\
$m=5$ & 95 (0.82)& 113 (0.85)& 113 (0.85) & 88 (0.81)& 121 (0.86)& 125 (0.86)\\
$m=8$ & 72 (0.77)& 82 (0.80)& 81 (0.80) & 67 (0.76)& 86 (0.81)& 88 (0.81)\\
$m=12$ & 57 (0.72)& 63 (0.74)& 62 (0.74) & 54 (0.71)& 65 (0.75)& 67 (0.76)\\
$m=16$ & 49 (0.68)& 52 (0.70)& 51 (0.69) & 46 (0.67)& 55 (0.71)& 56 (0.72)\\
$m=20$ & 44 (0.65)& 45 (0.66)& 44 (0.66) & 40 (0.63)& 48 (0.68)& 49 (0.68)\\
\hline
 & \multicolumn{3}{c||}{$\NCG = 445$, \textcolor{black}{$\NPCG = 326$}} 
 & \multicolumn{3}{c}{$\NCG = 633$, \textcolor{black}{$\NPCG = 480$}}\\  
\hline
 & \multicolumn{3}{c||}{Set 3} & \multicolumn{3}{c}{Set 4}\\ 
\hline
 & \multirow{2}{*}{TL}&\multicolumn{2}{c||}{W-cycle} & \multirow{2}{*}{TL}&\multicolumn{2}{c}{W-cycle}\\ 
 \hhline{~~--~--}
 & & 3 lvl & 4 lvl & &  3 lvl & 4 lvl\\
\hline
$m=3$ & 140 (0.88)& 188 (0.91)& 192 (0.91) & 162 (0.89)& 198 (0.91)& 198 (0.91)\\
$m=5$ & 99 (0.83)& 124 (0.86)& 128 (0.87) & 112 (0.85)& 131 (0.87)& 131 (0.87)\\
$m=8$ & 74 (0.78)& 89 (0.81)& 91 (0.82) & 83 (0.80)& 94 (0.82)& 94 (0.82)\\
$m=12$ & 58 (0.73)& 68 (0.76)& 69 (0.76) & 65 (0.75)& 73 (0.77)& 72 (0.77)\\
$m=16$ & 49 (0.68)& 56 (0.72)& 57 (0.72) & 55 (0.71)& 61 (0.74)& 61 (0.74)\\
$m=20$ & 43 (0.65)& 48 (0.68)& 49 (0.68) & 49 (0.68)& 53 (0.71)& 53 (0.70)\\
\hline
 & \multicolumn{3}{c||}{$\NCG = 946$, \textcolor{black}{$\NPCG = 678$}} 
 & \multicolumn{3}{c}{$\NCG = 1234$, \textcolor{black}{$\NPCG = 958$}}\\ 
\hline
\end{tabular}
\caption{Iteration counts and converge factor (in parenthesis) of the $h$-version of the two-level and W-cycle solvers and iteration counts of the \textcolor{black}{CG/PCG methods} as a function of $m$ ($C_\sigma^j \equiv C_\sigma=10$, $p=1$). \added{Sequences of agglomerated meshes shown in Figure~\ref{fig:grids}.}}
\label{tab:hMGvsm}
\end{table}
%%%%%%%%%%%%%%%%%%%%%%%
%%%%%%%%%%%%%%%%%%%%%%%%%
\begin{table}[t!]
\setlength\extrarowheight{1.5pt}
\centering
\footnotesize
\begin{tabular}{lccc||ccc||ccc}
\hline
 & \multicolumn{3}{c||}{Set 1} & \multicolumn{3}{c||}{Set 2} & \multicolumn{3}{c}{Set 3}\\ 
\hline
 & \multirow{2}{*}{TL}&\multicolumn{2}{c||}{W-cycle} & \multirow{2}{*}{TL}&\multicolumn{2}{c||}{W-cycle}& \multirow{2}{*}{TL}&\multicolumn{2}{c}{W-cycle} \\ 
 \hhline{~~--~--~--}
 & & 3 lvl & 4 lvl & &  3 lvl & 4 lvl & &  3 lvl & 4 lvl\\
\hline
$m=3$ & 1281 & 1334 & 1342 & 1168 & 1272 & 1362 & 1230 & 1379 & 1391\\
$m=5$ & 816 & 832 & 839 & 737 & 790 & 844 & 774 & 852 & 860\\
$m=8$ & 546 & 551 & 561 & 487 & 517 & 551 & 513 & 555 & 557\\
$m=12$ & 388 & 394 & 400 & 343 & 363 & 385 & 362 & 387 & 384\\
$m=16$ & 305 & 312 & 316 & 268 & 284 & 299 & 284 & 301 & 296\\
$m=20$ & 254 & 261 & 263 & 222 & 235 & 246 & 235 & 249 & 242\\
\hline
 & \multicolumn{3}{c||}{$\NCG = 1954$, \textcolor{black}{$\NPCG = 885$}} 
 & \multicolumn{3}{c||}{$\NCG = 2809$, \textcolor{black}{$\NPCG = 1264$}} 
 & \multicolumn{3}{c}{$\NCG = 4174$, \textcolor{black}{$\NPCG = 1708$}}\\ 
\hline
\end{tabular}
\caption{Iteration counts of the $h$-version of the two-level and W-cycle solvers \textcolor{black}{as a function of $m$  and the number of levels and iteration counts of the CG/PCG methods}  ($C_\sigma^j \equiv C_\sigma=10$, $p=3$). \added{Sequences of agglomerated meshes shown in Figure~\ref{fig:grids}.}}
\label{tab:hMGvsm3}
\end{table}
%%%%%%%%%%%%%%%%%%%%%%%%%%%%%%%

\added{Next, we consider the sets of nested grids obtained by agglomerating the shape-regular triangular meshes shown in Figure~\ref{fig:tria_grids}}, first row. The initial triangular decompositions constist of 528 (Set 1), 1086 (Set 2), 2198 (Set 3) and 4318 (Set 4) elements, cf. Figure~\ref{fig:tria_grids}, first row. In Table~\ref{tab:hMGvsm_tria}, we show the iteration counts needed to attain convergence with respect to a fixed tolerance of  $10^{-8}$ as a function of the set (\emph{i.e.}, the number of elements) and the number of smoothing steps of the $h$-version of the two-level and W-cycle multigrid solvers, with $p_j=p=1$. We recall that, as in the previous numerical test, we have considered $C_\sigma^j \equiv C_\sigma=10$, for each $j$. The results are similar to the case of initial polygonal meshes, with uniform convergence with respect to the granularity of the mesh and, in the case of the W-cycle solver, also with respect to the number of levels. We again attain improved performance, compared to the standard CG \textcolor{black}{and PCG methods}, in terms of the number of iterations required to attain convergence.
%%%%%%%%%%%%%%%%%%%%%%%%%%%%%%%%%%%%%%%
\begin{table}[t!]
\centering
\footnotesize
\setlength\extrarowheight{1.5pt}
\begin{tabular}{lccc||ccc}
\hline
 & \multicolumn{3}{c||}{Set 1} & \multicolumn{3}{c}{Set 2}\\ 
\hline
 & \multirow{2}{*}{TL}&\multicolumn{2}{c||}{W-cycle} & \multirow{2}{*}{TL}&\multicolumn{2}{c}{W-cycle}\\ 
 \hhline{~~--~--}
 & & 3 lvl & 4 lvl & &  3 lvl & 4 lvl\\
\hline
$m=4$ & 246 (0.90)& 258 (0.90)& 262 (0.90) & 282 (0.89)& 291 (0.90)& 292 (0.90)\\
$m=6$ & 177 (0.87)& 185 (0.87)& 188 (0.87) & 199 (0.86)& 205 (0.87)& 204 (0.87)\\
$m=10$ & 120 (0.81)& 125 (0.82)& 127 (0.82) & 133 (0.81)& 136 (0.82)& 136 (0.82)\\
$m=14$ & 94 (0.77)& 98 (0.78)& 99 (0.78) & 104 (0.77)& 106 (0.78)& 106 (0.78)\\
$m=18$ & 79 (0.74)& 82 (0.74)& 83 (0.74) & 87 (0.74)& 89 (0.75)& 89 (0.75)\\
\hline
 & \multicolumn{3}{c||}{$\NCG = 551$, \textcolor{black}{$\NPCG = 369$}} 
 & \multicolumn{3}{c}{$\NCG = 771$, \textcolor{black}{$\NPCG = 504$}}\\
 \hline
 & \multicolumn{3}{c||}{Set 3} & \multicolumn{3}{c}{Set 4}\\ 
\hline
 & \multirow{2}{*}{TL}&\multicolumn{2}{c||}{W-cycle} & \multirow{2}{*}{TL}&\multicolumn{2}{c}{W-cycle}\\ 
 \hhline{~~--~--}
 & & 3 lvl & 4 lvl & &  3 lvl & 4 lvl\\
\hline
$m=4$ & 328 (0.90)& 333 (0.91)& 329 (0.90) & 421 (0.91)& 425 (0.91)& 422 (0.91)\\
$m=6$ & 231 (0.87)& 234 (0.88)& 232 (0.87) & 292 (0.88)& 293 (0.89)& 292 (0.89)\\
$m=10$ & 153 (0.82)& 154 (0.83)& 153 (0.82) & 190 (0.83)& 191 (0.84)& 189 (0.84)\\
$m=14$ & 118 (0.78)& 119 (0.79)& 118 (0.78) & 145 (0.79)& 148 (0.80)& 146 (0.80)\\
$m=18$ & 98 (0.75)& 99 (0.75)& 98 (0.75) & 120 (0.76)& 123 (0.77)& 122 (0.77)\\
\hline
 & \multicolumn{3}{c||}{$\NCG = 1145$, \textcolor{black}{$\NPCG = 718$}} 
 & \multicolumn{3}{c}{$\NCG = 1630$, \textcolor{black}{$\NPCG = 974$}}\\ 
\hline
\end{tabular}
\caption{Iteration counts and converge factor (in parenthesis) of the $h$-version of the two-level and W-cycle solvers and iteration counts of the CG \textcolor{black}{and PCG methods} as a function of $m$ ($C_\sigma^j \equiv C_\sigma=10$, $p=1$). \added{Sequences of agglomerated meshes shown in Figure~\ref{fig:tria_grids}.}}
\label{tab:hMGvsm_tria}
\end{table}
%%%%%%%%%%%%%%%%%%%%%%%%%%%%%%%%%%%%%%

\added{Finally, we present a more exhaustive investigation of the effect of increasing $p$ while keeping fixed the number of smoothing steps. For this set of experiments, we consider a fine grid of 1024 elements and the corresponding agglomerated meshes (Set 2 in Figure~\ref{fig:grids}). In Table~\ref{tab:hMGvsp} we report the iteration counts of the $h$-version of the two-level and W-cycle solvers as a function of $p$, employing $m=5$ pre- and post- smoothing steps. We observe that, as expected, even though 
both multilevel solvers converge for a fixed value of $m$, the number of iterations required to
reduce the relative residual below the given tolerance grows with increasing $p$. 
However, the two-level and W-cycle multigrid solvers still employ
less iterations, than the number required by both the CG and PCG methods, cf. the last two columns of Table~\ref{tab:hMGvsp}.}
%%%%%%%%%%%%%%%%%%%%%%%%%%%%%%%
\begin{table}[t!]
\setlength\extrarowheight{1.5pt}
\centering
\footnotesize
\begin{tabular}{lccc|cc}
\hline
& \multirow{2}{*}{TL}
&\multicolumn{2}{c|}{W-cycle}
& \multirow{2}{*}{$\NCG$}
& \multirow{2}{*}{\textcolor{black}{$\NPCG$}}\\
\hhline{~~--~~}
 &  &  3 lvl & 4 lvl  &  &  \\
\hline
$p=1$ & 88 & 121 &125 & 633 & \textcolor{black}{480} \\
$p=2$ & 357 & 434 & 443 &1701   &\textcolor{black}{953} \\
$p=3$ & 737 & 790 & 844 &2809 & \textcolor{black}{1264} \\
$p=4$ & 958 & 1093 & 1184 & 4574 & \textcolor{black}{1821} \\
$p=5$ & 876 & 1096 & 1201 & 6796 &\textcolor{black}{2213} \\
\hline
\end{tabular}
\caption{Iteration counts of the $h$-version of the two-level and W-cycle solvers \textcolor{black}{as a function of $p$ and the number of levels and corresponding CG/PCG iteration counts} ($C_\sigma^j \equiv C_\sigma=10$, $m=5$). \added{Sequence of agglomerated meshes shown in Figure~\ref{fig:grids}, Set 2.}}
\label{tab:hMGvsp}
\end{table}
%%%%%%%%%%%%%%%%%%%%%%%%%%%%%%%%%%%%%
\added{As a numerical comparison, we have solved the correspoding linear systems of equations employing an unsmoothed aggregation Algebraic Multigrid (AMG) algorithm based on three popular algebraic agglomeration strategies.
More precisely, in the considered AMG methods the agglomerates are formed, at a purely algebraic level, by using either
the maximal independent set, the (approximate) maximum weighted matching or the greedy aggregation algorithms, and the resulting coarse levels are then employed as before 
within a $W$-cycle iteration with a Richardon smoother with $m=5$ pre- and post-smoothing steps. 
In Table~\ref{tab:AMG} we report, for each of the considered agglomeration strategies, 
the number of agglomeration levels (N. levels), as well as the computed convergence 
factors $\rho$. 
As it is clear from the results reported in Table~\ref{tab:AMG}, classical algebraic 
agglomeration procedures \paul{are not efficient when applied to} 
the matrices arising from high-order DG approximations; indeed in all the cases 
the resulting algorithm is not able to reduce the (relative) residual below the 
given tolerance 
within 5000 iterations, cf. last column of Table~\ref{tab:AMG} where the 
iteration counts ($\NAMG$) are shown. 
Such behaviour strongly suggests that the geometric information \paul{needs to}
 be taken into account 
in the construction of the solver and/or more sophisticated (aggressive) 
aggregation-based algebraic 
algorithms, as well as Schwarz-type smoothers, such as the ones proposed, for example, in
\cite{OlsonSchroder_2011,BastianBlattScheichl_2012}, should be considered. 
Such developments are currently under investigation and will be the subject of future research.}
%%%%%%%%%%%%%%%%%%%%%%%%%%%%%%%
\begin{table}[t!]
\setlength\extrarowheight{1.5pt}
\centering
\footnotesize
\begin{tabular}{c|cc|cc|cc|c}
& \multicolumn{2}{|c|}{Max. independent set}
& \multicolumn{2}{c|}{Max. weighted matching}
& \multicolumn{2}{c|}{Greedy aggregation}
&\\
%\hhline{~------~}  
&$\textcolor{black}{\rho}$  & \textcolor{black}{N. levels}
&  $\textcolor{black}{\rho}$  & \textcolor{black}{N. levels}
&  $\textcolor{black}{\rho}$  & \textcolor{black}{N. levels}
& $\NAMG$\\
\hline  

$p=1$&\textcolor{black}{0.9990} & \textcolor{black}{3} 
&\textcolor{black}{0.9987} & \textcolor{black}{7} 
&\textcolor{black}{0.9992} & \textcolor{black}{5} 
& \textcolor{black}{$>$ 5000}   \\
$p=2$&\textcolor{black}{0.9989} & \textcolor{black}{3} 
&\textcolor{black}{ 0.9986} & \textcolor{black}{8} 
&\textcolor{black}{0.9988} & \textcolor{black}{5} 
& \textcolor{black}{$>$ 5000}   \\
$p=3$&\textcolor{black}{0.9989} & \textcolor{black}{3} 
&\textcolor{black}{0.9986} & \textcolor{black}{8} 
&\textcolor{black}{0.9990} & \textcolor{black}{6} 
& \textcolor{black}{$>$ 5000}   \\
$p=4$&\textcolor{black}{0.9989} & \textcolor{black}{3} 
&\textcolor{black}{0.9986} & \textcolor{black}{9} 
&\textcolor{black}{0.9989} & \textcolor{black}{6} 
& \textcolor{black}{$>$ 5000}   \\
$p=5$&\textcolor{black}{0.9989} & \textcolor{black}{3} 
&\textcolor{black}{0.9986} & \textcolor{black}{9} 
&\textcolor{black}{0.9988} & \textcolor{black}{6} 
& \textcolor{black}{$>$ 5000}   \\
\hline
\end{tabular}
\caption{\added{Number of agglomeration levels (N. levels) and computed convergce factors ($\rho$) of Algebraic $W$-cycle multigrid method (Richardon smoother, $m=5$ pre- and post-smoothing steps) as a function of $p$. The agglomerates are formed algebraically by using either the maximal independent set, the (approximate) maximum weighted matching or the greedy aggregation algorithms. The initial fine grid is shown in Figure~\ref{fig:grids}, Set 2 (first row).}}
\label{tab:AMG}
\end{table}
%%%%%%%%%%%%%%%%%%%%%%%%%%%%%%%%%%%%%%%%%%%%
%%%%%%%%%%%%%%%%%%%%%%%%%%%%%%%%%%%%%%%%%
%%%%%%%%%%%%%%%%%%%%%%%%%%%%%%%%%%%%%%%%%
\FloatBarrier
%%%%%%%%%%%%%%%%%%%%%%%%%%%%%%%%%%%%%%%

\section{Conclusions}
\label{sec:conclusions}
\added{
We have presented and analyzed two-level and multigrid schemes for the efficient solution of the linear system of equations arising from the $hp$-version of the interior penalty DG scheme on polygonal/polyhedral meshes. 
The attractive feature of the proposed algorithms is that the auxiliary sequence of meshes needed 
by the multilevel solver can be generated by a (successive) general geometric agglomeration procedure 
starting from an initial grid made of (possibly arbitrarily-shaped) elements. 
Such an approach fully exploits the flexibilty of DG methods \paul{in terms of their ability 
to} handle arbitrarily-shaped elements, including polytopic elements, see \cite{AntBreMar09,dg_cfes_2012,Canetal14,BasBotColSub,CangianiDongGeorgoulisHouston_2016,AntFacRusVer2016,CangianiDongGeorgoulis_2016}, and the recent review paper \cite{Anetal16_review}.
Extending the theoretical results recently presented in \cite{AntSarVer15} on quasi-uniform meshes, we have proved that, under mild geometric assumptions on the quality of the agglomerates, both the two-level and W-cycle multigrid schemes converge uniformly with respect to the discretization parameters (namely, the granularity of the underlying partition and the polynomial approximation degree $p$) and, for the multigrid scheme, the number of levels, provided that the number of smoothing steps is chosen \textcolor{black}{sufficiently large}.  We have also demonstrated through numerical experiments that the theoretical assumption \paul{concerning the need to employ a} sufficently large number of smoothing steps is not needed in practice, i.e., our algorithms converge even if the number of smoothing steps is kept fixed independently of the polynomial approximation degree $p$. \paul{However, in this latter} case, the performance of the iterative solvers deteriorates, as expected, when increasing $p$.
}

\appendix
%%%%%%%%%%

\end{document}